\documentclass[graybox]{svmult}

\usepackage{mathptmx}       
\usepackage{helvet}         
\usepackage{courier}        
\usepackage{type1cm}        
\usepackage{makeidx}         
\usepackage{graphicx, enumerate}        
\usepackage{multicol}
\usepackage[caption=false]{subfig}        
\usepackage[bottom]{footmisc}
\usepackage{tikz,xcolor,amsmath,amssymb,mathtools,tikz-cd}
\usetikzlibrary{quotes}
\spdefaulttheorem{project}{Research Project}{\bf}{}
\newcommand{\resproject}[1]{\begin{svgraybox} \begin{project} #1 \end{project} \end{svgraybox}}

\spdefaulttheorem{cproblem}{Challenge Problem}{\bf}{}

\newcommand{\cok}{{\operatorname{cok}}}
\newcommand{\image}{{\operatorname{Im}}}
\newcommand{\Div}{{\operatorname{Div}}}
\newcommand{\Aut}{{\operatorname{Aut}}}
\newcommand{\Sym}{{\operatorname{Sym}}}

\newcommand{\Null}{{\operatorname{Null}}}
\newcommand{\gon}{{\operatorname{gon}}}

\newcommand{\ord}{{\operatorname{ord}}}

\newcommand{\divi}{{\operatorname{div}}}

\newcommand{\Prin}{{\operatorname{Prin}}}
\newcommand{\Hom}{{\operatorname{Hom}}}
\newcommand{\legen}[2]{\left(\frac{#1}{#2}\right)}

\newcommand{\diag}{{\operatorname{diag}}}

\newcommand*{\isoarrow}[1]{\arrow[#1,"\rotatebox{90}{\(\sim\)}"]}

\newcommand{\vx}{\mathbf{x}}
\newcommand{\vone}{\mathbf{1}}
\newcommand{\vzero}{\mathbf{0}}
\newcommand{\vr}{\mathbf{r}}
\newcommand{\vd}{\mathbf{d}}
\newcommand{\ZZ}{\mathbb{Z}}
\newcommand{\QQ}{\mathbb{Q}}

\newcommand{\C}{\mathbb{C}}

\newcommand{\CT}{\mathcal{T}}
\newenvironment{sysmatrix}[1]
 {\left(\begin{array}{@{}#1@{}}}
 {\end{array}\right)}
\newcommand{\ro}[1]{%
  \xrightarrow{\mathmakebox[\rowidth]{#1}}%
}
\newlength{\rowidth}
\AtBeginDocument{\setlength{\rowidth}{3em}}

\spnewtheorem*{prerequisites}{Suggested prerequisites}{\bf}{\small\em}

\makeindex             

\newcommand{\ugrad}[1]{{\textcolor{red}{#1}}}

\begin{document}

\title*{Chip-Firing Games and Critical Groups}
\author{Darren Glass and Nathan Kaplan}
\institute{Darren Glass\at Gettysburg College, 300 N Washington St, Gettysburg, PA 17325, \email{dglass@gettysburg.edu}
\and Nathan Kaplan \at University of California, 340 Rowland Hall, Irvine, CA 92697 \email{nckaplan@math.uci.edu}}
%
%
\maketitle

\abstract{In this note we introduce a finite abelian group that can be associated to any finite connected graph.  This group can be defined in an elementary combinatorial way in terms of chip-firing operations, and has been an object of interest in combinatorics, algebraic geometry, statistical physics, and several other areas of mathematics.  We will begin with basic definitions and examples and develop a number of properties that can be derived by looking at this group from different angles.  Throughout, we will give exercises, some of which are straightforward and some of which are open questions.  We will also attempt to highlight some of the many contributions to this area made by undergraduate students.}

\begin{prerequisites}
The basic definitions and themes of this note should be accessible to any student with some knowledge of linear algebra and group theory.  As we go along, deeper understanding of graph theory, abstract algebra, and algebraic geometry will be of use in some sections.
\end{prerequisites}

\section{Critical Groups}\label{sec1}

The primary object of interest in this chapter will be a finite abelian group that is associated to a graph.  This group has been studied from a variety of different perspectives, and as such it goes by several different names, including the {\it sandpile group}, the {\it component group}, the {\it critical group}, or the {\it Jacobian} of a graph.  We will give definitions and some results about critical groups of graphs and pose some questions that we think would be interesting for an undergraduate to tackle.  For additional background and motivation for this topic as well as a more in-depth treatment, we recommend the books by Klivans \cite{Klivans} and Corry \& Perkinson \cite{CP}.

We will highlight several significant contributions to the study of critical groups made by undergraduates -- papers with at least one undergraduate author are highlighted in red in the bibliography -- and we will discuss some open problems that would make excellent topics for future undergraduate research.

\subsection{Definitions and Examples}

Part of what makes the study of critical groups such a good topic for undergraduate research is that the definitions are very concrete and one can get started computing examples right away.

Let $G$ be a connected, undirected graph with vertex set $V(G)$ of finite size $n$ and edge set $E(G)$. Choose an ordering of $V(G)\colon v_1,\ldots, v_n$.  We define the {\it adjacency matrix} of the graph $G$ to be the $n \times n$ matrix $A$ where the entry $a_{i,j}$ in the $i$\textsuperscript{th} row and $j$\textsuperscript{th} column of $A$ is the number of edges between $v_i$ and $v_j$.  We also define the matrix $D$ to be the diagonal matrix where the entry $d_{i,i}$ is equal to the degree of $v_i$.  Finally, we let $L(G)$ be the matrix $D-A$; this matrix is referred to as the \index{Laplacian matrix} {\it Laplacian matrix}, or {\it combinatorial Laplacian}, of the graph $G$.  We often write $L$ for this matrix when the graph is clear from context.

\smallskip

\noindent{\textbf{Note}}: We defined the adjacency matrix $A$ of $G$ by saying that $a_{i,j}$ is the \emph{number of edges} between $v_i$ and $v_j$, implying that this number can be greater than $1$.  For most of this paper we focus on the case of \emph{simple} graphs (at most one edge between any pair of vertices)\index{simple graphs}, with no \emph{self-loops} (edges from $v_i$ to $v_i$), that are \emph{connected}\index{connected graphs} (for any pair of vertices $v_i, v_j$ there is a path from $v_i$ to $v_j$ in $G$), and where edges are \emph{undirected}.  In this case we will denote an edge between $v_i$ and $v_j$ as $\overline{v_i v_j}$.  Much of the theory of critical groups carries over to more general settings, but we find that it is most helpful to first focus on this simplest case.

\begin{example}\label{E:run}

We will consider the graph below:
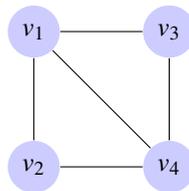
\begin{center}
\begin{tikzpicture}
  [scale=.3,auto=left,every node/.style={circle,fill=blue!20}]
  \node (l1) at (1,7) {$v_1$};
  \node (l2) at (1,1) {$v_2$};
  \node (r1) at (7,7)  {$v_3$};
  \node (r2) at (7,1)  {$v_4$};
\foreach \from/\to in {l1/r1,l2/r2,l1/l2,r1/r2,l1/r2}
    \draw (\from) -- (\to);
\end{tikzpicture}
\end{center}
One can see that the adjacency matrix, degree matrix, and Laplacian of this graph are given by:
\[
A = \left(
                                                                       \begin{array}{cccc}
                                                                         0 & 1 & 1 & 1 \\
                                                                         1 & 0 & 0 & 1 \\
                                                                         1 & 0 & 0& 1 \\
                                                                         1 & 1 & 1 & 0\\
                                                                       \end{array}
                                                                     \right),\ \ \ D = \left(
                                                                                                                       \begin{array}{cccc}
                                                                                                                         3 & 0 & 0 & 0 \\
                                                                                                                         0 & 2 & 0 & 0 \\
                                                                                                                         0 & 0 & 2 & 0 \\
                                                                                                                         0 & 0 & 0 & 3 \\
                                                                                                                       \end{array}
                                                                                                                     \right),\ \ \
                                                                     L = \left(
                                                                                                                       \begin{array}{cccc}
                                                                                                                         3 & -1 & -1 & -1 \\
                                                                                                                         -1 & 2 & 0 & -1 \\
                                                                                                                         -1 & 0 & 2 & -1 \\
                                                                                                                         -1 & -1 & -1 & 3 \\
                                                                                                                       \end{array}
                                                                                                                     \right).\]

\end{example}

It follows from the definition of the Laplacian matrix of a graph that the entries in any row or in any column sum to $0$. This implies that the vector consisting of all ones, $\vone$, is in the null space of the matrix.  In fact, we have the following result:

\begin{theorem}\label{thm_nullspace}
For any finite connected graph $G$, the null space of the Laplacian matrix of $G$ is generated by the vector $\vone$.
\end{theorem}

\begin{proof}
Since $\vone$ is in the null space, all multiples of it are as well.  Let $\vx = (x_1,\ldots, x_n)$ be a vector in the null space of $L$, so that $L\vx=\vzero$, the all zero vector.  Note that this implies that $\vx^T L\vx = 0$.  One can check that
\[
\vx^T L\vx = \sum_{\overline{v_i v_j} \in E(G)} (x_i-x_j)^2.
\]
Each of these terms is nonnegative so the entries of $\vx$ corresponding to any pair of neighboring vertices must be equal.  Because $G$ is connected we must have that for any vector in the null space all of the entries in $\vx$ are equal, concluding the proof.
\end{proof}

More generally, we can determine the number of connected components of $G$ in terms of its Laplacian.
\begin{proposition}
For any finite graph $G$, the dimension of the null space of the Laplacian matrix of $G$ is the number of connected components of $G$.
\end{proposition}

\begin{exercise}
If $G$ is a graph with $c$ connected components, describe $c$ linearly independent vectors in the null space of $L$.  Mimic the proof of Theorem \ref{thm_nullspace} to show that the dimension of the null space is, in fact, $c$.
\end{exercise}

This result is the first of many results relating the Laplacian matrix of a graph to other seemingly combinatorial properties of the graph.  The eigenvalues of the Laplacian turn out to be particularly interesting, and the area of \emph{spectral graph theory} is largely dedicated to studying this relationship.  We refer the interested reader to the survey article \cite{Spielman} or the book \cite{Chung}.

In order to discuss our main object of interest, we note that any $n \times n$ integer matrix $A$ can be thought of as a linear map $A\colon \ZZ^n \to \ZZ^n$.  The {\it cokernel} of $A$, denoted $\cok(A)$, is $\ZZ^n/\image(A)$.  Theorem \ref{thm_nullspace} implies that if $L$ is the Laplacian of a connected graph $G$ then $\dim(\image(L)) = n-1$, so $\cok(L) \cong \ZZ \oplus K$ for some finite abelian group $K$.  This group $K$ is the {\it critical group} \index{critical group} of the graph $G$.  We will denote it by either $K$ or $K(G)$ depending on whether the graph is understood by context.

The main goal of this article is to outline problems about critical groups.  What interesting information does $K(G)$ tell us about $G$?  In Section \ref{sec:spanning} we will see that the order of $K(G)$ tells us about the subgraphs of $G$, in particular, that $|K(G)|$ is the number of spanning trees of $G$.  In the next section we will introduce divisors on $G$ and see that the structure of the finite abelian group $K(G)$ tells us something about how these divisors on $G$ behave under chip-firing operations.

\subsection{Divisors on a Graph and the Chip-Firing Game}\label{sec:divisors}

We started by giving an algebraic description of the critical group as the torsion part of the cokernel of the Laplacian matrix of $G$, but one can also approach it from a more combinatorial point of view via the \index{chip-firing game} {\it chip-firing game}, which was originally introduced by Biggs in \cite{Biggs}.  In order to define this game, we set some notation.  A \index{divisors on graphs}{\it divisor} on a graph $G$ is a function $\delta \colon V(G) \rightarrow \ZZ$, which we think of as assigning an integer number of chips to each vertex of $G$.  We can think of a divisor as an element of $\ZZ^{|V(G)|}$.  The \emph{degree}\index{degree} of a divisor is defined by $\deg(\delta) = \sum_{v}\delta(v)$.  We define an addition of divisors by $(\delta_1 + \delta_2)(v)=\delta_1(v)+\delta_2(v)$.  In this way, we see that the set of all divisors on $G$, denoted $\Div(G)$, is isomorphic to a free abelian group with $|V(G)|$ generators.  We let $\Div^0(G)$ denote the subgroup of all degree $0$ divisors on $G$.  One can see that $\Div^0(G)$ is isomorphic to a free abelian group with $|V(G)|-1$ generators.

\begin{exercise}
Describe a set $\Delta$ of $|V(G)|-1$ divisors on $G$ so that $\Div^0(G)$ is isomorphic to the free abelian group on $\Delta$.
\end{exercise}

We next define two types of transitions between divisors, which are called \emph{chip-firing moves}\index{chip-firing moves}.  In the first, we choose a vertex and {\it borrow} a chip from each of its neighbors.  The second is an inverse to the first, where we choose a vertex and {\it fire} it, sending a chip to each one of its neighbors.  We will treat these two as inverses in an algebraic sense, so for example when we say `perform $-2$ borrowings at $v$' one should think of it as the same as `perform $2$ firings at $v$'  Note that each one of these chip-firing moves preserves the degree of a divisor.  Two divisors $D_1$ and $D_2$ are {\it equivalent} if we can get from $D_1$ to $D_2$ by a sequence of chip-firing moves.

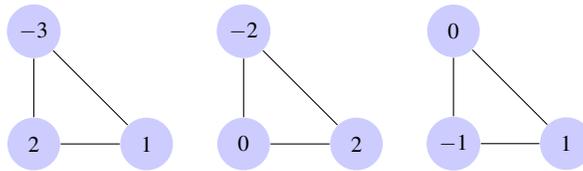
\begin{figure}
\begin{center}
\begin{tikzpicture}
  [scale=.5,auto=left,every node/.style={circle,fill=blue!20,minimum size=.7cm}]
  \node (l1) at (0,3) {$-3$};
  \node (l2) at (0,0) {$2$};
  \node (r1) at (3,0)  {$1$};

\foreach \from/\to in {l1/l2,l2/r1,r1/l1}
    \draw (\from) -- (\to);
\end{tikzpicture}
\hskip .2in
\begin{tikzpicture}
  [scale=.5,auto=left,every node/.style={circle,fill=blue!20,minimum size=.7cm}]
  \node (l1) at (0,3) {$-2$};
  \node (l2) at (0,0) {$0$};
  \node (r1) at (3,0)  {$2$};
\foreach \from/\to in {l1/l2,l2/r1,r1/l1}
    \draw (\from) -- (\to);
\end{tikzpicture}
\hskip .2in
\begin{tikzpicture}
  [scale=.5,auto=left,every node/.style={circle,fill=blue!20,minimum size=.7cm}]
  \node (l1) at (0,3) {$0$};
  \node (l2) at (0,0) {$-1$};
  \node (r1) at (3,0)  {$1$};
\foreach \from/\to in {l1/l2,l2/r1,r1/l1}
    \draw (\from) -- (\to);
\end{tikzpicture}
\caption{A divisor on the cycle graph $C_3$, followed by the divisor obtained by first `firing' at the lower-left vertex and then `borrowing' at the upper-left vertex}
\end{center}
\end{figure}

The set of divisors that are equivalent to the all zero divisor is exactly $\image(L(G))$.  Starting with a divisor $\delta$, which we think of a column vector in $\ZZ^{|V(G)|}$, firing $v_i$ corresponds to subtracting the $i$\textsuperscript{th} column of $L(G)$ from this vector.  Similarly, borrowing at $v_i$ corresponds to adding $i$\textsuperscript{th} column of $L(G)$. This gives a second interpretation of the critical group.

\begin{proposition}
Let $G$ be a finite connected graph.  The critical group $K(G)$ is isomorphic to $\Div^0(G)/\sim$, the set of all degree $0$ divisors of $G$ modulo chip-firing equivalence.
\end{proposition}

\begin{example}
Let $G$ be the cycle on three vertices.  Consider any divisor $\delta$ of degree zero on $G$.  Let $\hat{\delta}$ be the divisor attained after performing $\delta(v_3)$ borrowing operations at $v_1$, so in particular $\hat{\delta}(v_3)=0$.  Because the degree of $\hat{\delta}$ is zero we must have that $\hat{\delta}(v_2)=-\hat{\delta}(v_1)$ so in particular $\hat{\delta}$ is a multiple of the divisor $\delta_{1,2}$ which is defined by setting $\delta_{1,2}(v_1)=1, \delta_{1,2}(v_2)=-1,$ and $\delta_{1,2}(v_3)=0$.  This implies that every element of $\Div^0(G)$ is equivalent to a multiple of $\delta_{1,2}$.  Therefore, $K(G)$ is cyclic.  One can also show that $3\delta_{1,2}$ is chip-firing equivalent to the zero divisor, but that $\delta_{1,2}$ and $2\delta_{1,2}$ are not.  We conclude that $K(C_3) \cong \ZZ/3\ZZ$.
\end{example}

\begin{remark}\label{remark_AG}
These definitions are in parallel with a family of ideas in algebraic geometry, and many recent results in the field have come from trying to better understand this analogy.  In particular, given a curve $C$ defined as the solution set to a polynomial equation $f(x,y)=0$, algebraic geometers define a {\it divisor} on the curve to be a formal finite linear combination $\sum a_i P_i$ of points on the curve. The degree of the divisor is defined to be the sum $\sum a_i$, and the set of divisors of degree zero is denoted by $\Div^0(C)$.  The {\it Jacobian} of the curve is then defined to be $\Div^0(C)/\sim$, where two divisors $\delta_1$ and $\delta_2$ are said to be equivalent if there is a rational function on $C$ whose zeroes are represented by $\delta_1$ and whose poles are represented by $\delta_2$.  For more details about Jacobians in algebraic geometry, we recommend \cite{GH}.
\end{remark}

\begin{exercise}
Show that if $\delta$ is a divisor of degree zero on the graph from Example \ref{E:run}, then $\delta$ is equivalent after some number of firing/borrowing operations to a divisor $\hat{\delta}$ so that $\hat{\delta}(v_3)=\hat{\delta}(v_4)=0$.  This result implies that every divisor of degree zero is equivalent to a multiple of the divisor $\delta_{1,2}$ which is defined by setting $\delta_{1,2}(v_1)=1, \delta_{1,2}(v_2)=-1,$ and $\delta_{1,2}(v_3)=\delta_{1,2}(v_4)=0$.

Next, show that the order of $\delta_{1,2}$ in $K(G)$ is $8$, proving that the critical group of this graph is $\ZZ/8\ZZ$.
\end{exercise}

\subsection{Smith Normal Forms}\label{SS:SNF}

We have defined the critical group of a connected graph $G$ as the torsion part of the cokernel of the Laplacian matrix of $G$, but it is not so clear how to determine the structure of this finite abelian group.  Linear algebra provides a nice solution.

\begin{proposition}\label{SNF_coker}
Let $L$ be a $n\times n$ integer matrix of rank $r$.  There exist matrices $U$ and $V$ with integer entries so that $\det(U)=\pm \det(V)= \pm 1$ and $S = ULV$ is a diagonal matrix where $s_{r+1,r+1} = s_{r+2,r+2} = \cdots = s_{n,n} = 0$ and $s_{i,i} \mid s_{i+1,i+1}$ for all $1 \le i <r$. The matrix $S$ is called the \index{Smith Normal Form}{\it Smith Normal Form} of $L$.

Moreover,
\[
\cok(L) \cong \cok(S) \cong \left(\ZZ/s_{1,1}\ZZ\right) \oplus \left(\ZZ/s_{2,2}\ZZ\right) \oplus \cdots \oplus \left(\ZZ/s_{r,r}\ZZ\right) \oplus \ZZ^{n-r}.
\]
\end{proposition}

In particular, one can read off the critical group of $G$ directly from the Smith normal form of $L(G)$.  The hard part here is showing the existence of the invertible matrices $U$ and $V$.  For a proof see \cite[Theorem 2.33]{CP}.  Once one knows that $U$ and $V$ satisfying these properties exist, the fact that the cokernels are isomorphic follows from the commutative diagram below.  Note that the fact that $U$ and $V$ have determinant $\pm 1$ means that they define isomorphisms $\ZZ^n \rightarrow \ZZ^n$.

\[ \begin{tikzcd}
1 \arrow{r} &\ZZ^{n-r} \isoarrow{d} \arrow{r} &\ZZ^n \arrow{r}{S} \arrow{d}{U} & \ZZ^n \arrow{r}  & \cok(S) \arrow{r} & 1  \\%
1 \arrow{r} &\ZZ^{n-r} \arrow{r} &\ZZ^n \arrow{r}{L} & \ZZ^n \arrow{r} \arrow{u}{V} & \cok(L) \arrow{r} & 1
\end{tikzcd}
\]

Finally, it is straightforward to determine the cokernel of a diagonal matrix, so the last claim follows.

\begin{example}
Consider the graph $G$ below:
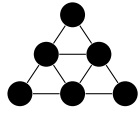
\begin{center}
\begin{tikzpicture} [scale=.7,auto=left,every node/.style={circle,fill=black}]

\node (a2) at (1,0) {};
\node (a3) at (2,0) {};
\node (a4) at (3,0) {};

\node (b2) at (1.5,.75) {};
\node (b3) at (2.5,.75){};

\node (c2) at (2,1.5) {};

\foreach \from/\to in {a2/a3,a3/a4,b2/b3,a2/b2,b2/c2,a3/b3,a4/b3,b3/c2,a3/b2}
    \draw (\from) -- (\to);
\end{tikzpicture}
\end{center}

We can see that
\[
L(G) = \left(
\begin{array}{cccccc}
 2 & -1 & 0 & -1 & 0 & 0 \\
 -1 & 4 & -1 & -1 & -1 & 0 \\
 0 & -1 & 2 & 0 & -1 & 0 \\
 -1 & -1 & 0 & 4 & -1 & -1 \\
 0 & -1 & -1 & -1 & 4 & -1 \\
 0 & 0 & 0 & -1 & -1 & 2 \\
\end{array}
\right),\]
and can write
\[ULV = \left(
\begin{array}{cccccc}
 0 & -1 & 0 & 0 & 0 & 0 \\
 0 & 0 & -1 & 0 & 0 & 0 \\
 0 & 0 & 1 & 0 & -1 & 0 \\
 0 & 0 & 1 & 0 & -1 & 1 \\
 1 & 2 & 3 & 0 & 4 & -7 \\
 1 & 1 & 1 & 1 & 1 & 1 \\
\end{array}
\right) L \left(
\begin{array}{cccccc}
 1 & 4 & -1 & 10 & 10 & 1 \\
 0 & 1 & 0 & 2 & 3 & 1 \\
 0 & 0 & 0 & 1 & 2 & 1 \\
 0 & 0 & 1 & -3 & -1 & 1 \\
 0 & 0 & 0 & 0 & 1 & 1 \\
 0 & 0 & 0 & 0 & 0 & 1 \\
\end{array}
\right) = \left(
\begin{array}{cccccc}
 1 & 0 & 0 & 0 & 0 & 0 \\
 0 & 1 & 0 & 0 & 0 & 0 \\
 0 & 0 & 1 & 0 & 0 & 0 \\
 0 & 0 & 0 & 3 & 0 & 0 \\
 0 & 0 & 0 & 0 & 18 & 0 \\
 0 & 0 & 0 & 0 & 0 & 0 \\
\end{array}
\right) =S.\]

In particular, $U$ and $V$ both have determinant $-1$, so $S$ is the Smith Normal Form of $L$.  This implies that the critical group of the graph is $\ZZ/3\ZZ \oplus \ZZ/18\ZZ$.

\end{example}

How do we actually compute the Smith Normal Form of a matrix? One useful fact (see, for example, \cite[Theorem 2.4]{SNF}) is the following:

\begin{theorem}\label{T:detgcd}
Let $L$ be an $n \times n$ integer matrix of rank $r$ whose Smith normal form has nonzero diagonal entries $s_1,\ldots, s_r$ where $s_i \mid s_{i+1}$ for all $1 \le i < r$. For each $i \le r$, we have that $s_1 s_2 \cdots s_i$ is equal to the greatest common divisor of all $i \times i$ minors of~$L$.
\end{theorem}

\begin{example}\label{ExampleKn}
Consider the complete graph $K_n$ on $n$ vertices.  One sees that
\[
L(K_n)=\left(
                                                                                                                     \begin{array}{ccccc}
                                                                                                                       n-1 & -1 & -1 & \cdots & -1 \\
                                                                                                                       -1 & n-1 & -1 & \cdots & -1 \\
                                                                                                                       -1 & -1 & n-1 & \cdots & -1 \\
                                                                                                                       \vdots & \vdots & \vdots & \ddots & \vdots \\
                                                                                                                       -1 & -1 & -1 & \cdots & n-1 \\
                                                                                                                     \end{array}
                                                                                                                   \right).
                                                                                                                   \]
                                                                                                                  The greatest common divisor of the entries of this matrix is $1$, so $s_1=1$. The $2 \times 2$ submatrices of this matrix are all of the following form:
\begin{align*}
& \left(
    \begin{array}{cc}
      -1 & -1 \\
      -1 & -1 \\
    \end{array}
  \right), \left(
    \begin{array}{cc}
      n-1 & -1 \\
      -1 & -1 \\
    \end{array}
  \right),\left(
    \begin{array}{cc}
      -1 & -1 \\
      n-1 & -1 \\
    \end{array}
  \right), \\
 & \left(
    \begin{array}{cc}
      -1 & -1 \\
      -1 & n-1 \\
    \end{array}
  \right),\left(
    \begin{array}{cc}
      -1 & n-1 \\
      -1 & -1 \\
    \end{array}
  \right),\left(
    \begin{array}{cc}
      n-1 & -1 \\
      -1 & n-1 \\
    \end{array}
  \right).
\end{align*}
In particular, the $2\times 2$ minors are all in the set $\{0,\pm n,n^2-2n\}$, and the greatest common divisor of these values is $n$.  This implies $s_2=n$, which in turn tells us that $n \mid s_i$ for all $2 \le i \le n-1$.  The determinant of the $(n-1) \times (n-1)$ matrix that we get by deleting the last row and column of $L(K_n)$ is $n^{n-2}$.  We conclude that $s_i=n$ for each $2 \le i \le n-1$.  This implies that the Smith Normal Form of the Laplacian is
\[
S=\left(
                                                                                                                     \begin{array}{cccccc}
                                                                                                                       1 & 0 & 0 & \cdots & 0 & 0\\
                                                                                                                       0 & n & 0 & \cdots & 0 & 0\\
                                                                                                                       0 & 0 & n & \cdots & 0 & 0\\
                                                                                                                       \vdots & \vdots & \vdots & \ddots & \vdots & \vdots \\
                                                                                                                       0 & 0 & 0 & \cdots &  n &0 \\
                                                                                                                       0 & 0 & 0 & \cdots & 0 & 0\\
                                                                                                                     \end{array}
                                                                                                                   \right)
                                                                                                                   \]
                                                                                                                   and therefore the critical group of the complete graph is $(\ZZ/n\ZZ)^{n-2}$.

\end{example}

\begin{exercise}
Verify that the determinant of the $(n-1) \times (n-1)$ matrix that we get by deleting the last row and column of $L(K_n)$ is $n^{n-2}$.
\end{exercise}

Theorem \ref{T:detgcd} gives an explicit (if not very effective) way to compute the Smith Normal Form, and thus the critical group, of any graph by computing many determinants of submatrices and their greatest common divisors.  However, it can also be used in other ways to tell us about the structure of the critical group.  For example, using the notation from Proposition \ref{SNF_coker}, if $G$ is a connected graph with $n$ vertices, then the product $s_1 \cdots s_{n-2}$ is the greatest common divisor of the $(n-2) \times (n-2)$ minors of $L(G)$.  So if any one of these minors is equal to $1$, then $s_1 \cdots s_{n-2} = 1$ and $|K(G)| = s_{n-1}$.  This gives the following result:
\begin{corollary}\label{C:cyclic}
Let $G$ be a connected graph on $n$ vertices.  If there exists an $(n-2) \times (n-2)$ minor of $L$ equal to $1$ then the critical group of $G$ is cyclic.
\end{corollary}

We have defined the critical group of a connected graph as the torsion part of the cokernel of the Laplacian matrix, but it is often convenient to think of the critical group as the cokernel of an invertible matrix.  Let the \emph{reduced Laplacian}\index{reduced Laplacian} of a connected graph $G$ be the matrix $L_0(G)$ (or just $L_0$ when the graph is clear from context) that we get from deleting the final row and column of $L(G)$.  Because all of the rows and columns of $L$ sum to $0$, the torsion part of $\cok(L)$ is equal to $\cok(L_0)$.  In fact, it is a special property of Laplacian matrices that one can remove any row and column from $L$ and the cokernels of the matrices will be isomorphic.  See \cite[Section 2.2.1]{CP} or \cite[Chapter 6]{BiggsBook} for more detail.  The following result then follows from Theorem \ref{T:detgcd}.

\begin{corollary}\label{C:order}
Let $G$ be a graph on $n$ vertices.  For any $i,j$ satisfying $1 \le i,j \le n$, let $L^{i,j}$ be the $(n-1) \times (n-1)$ matrix that we get by deleting the $i$\textsuperscript{th} row and $j$\textsuperscript{th} column of $L(G)$.  Then $K(G) \cong \cok(L^{i,j})$.  In particular, the order of $K(G)$ is equal to the determinant of the reduced Laplacian $L^{i,j}$.
\end{corollary}

As mentioned earlier, the algorithm suggested by Theorem \ref{T:detgcd} is not very efficient. There are much more efficient algorithms for computing Smith Normal Forms that proceed similarly to how one row-reduces matrices into reduced echelon form in a linear algebra class.  In particular, one can put any $n \times n$ integer matrix into a unique matrix in Smith Normal Form by a sequence of the following operations:
\begin{enumerate}
\item Multiply rows or columns by $-1$,
\item Swap two rows,
\item Swap two columns,
\item Add any integer multiple of one row to another row, or
\item Add any integer multiple of one column to another column.
\end{enumerate}

Just as when putting matrices into reduced echelon form there are many choices one makes along the way which may speed up or slow down the process.  For details of how to optimize this procedure, we refer the reader to \cite{Gies} and \cite{Stor}.  There are efficient implementations of these algorithms in most computer algebra systems including Sage, Maple, and Mathematica.

\begin{example}
Consider again the graph from Example \ref{E:run}.  Let us use row and column reduction in order to find the Smith Normal Form of the Laplacian of this graph.
\begin{alignat*}{2}
\begin{sysmatrix}{rrrr}
  3 & -1 & -1 & -1 \\
                                                                                                                         -1 & 2 & 0 & -1 \\
                                                                                                                         -1 & 0 & 2 & -1 \\
                                                                                                                         -1 & -1 & -1 & 3 \\
\end{sysmatrix}
&\!\begin{aligned}
&\ro{r_1 \leftrightarrow r_2}\\
&\ro{r_4-r_3}
\end{aligned}
\begin{sysmatrix}{rrrr}
-1 & 2 & 0 & -1 \\
3 & -1 & -1 & -1 \\
-1 & 0 & 2 & -1 \\
0 & -1 & -3 & 4
\end{sysmatrix}
&\!\begin{aligned}
&\ro{r_2 +3r_1}\\
&\ro{r_3-r_1}
\end{aligned}
\begin{sysmatrix}{rrrr}
-1 &  2 &   0 & -1 \\
0 &  5 & -1 & -4 \\
0 & -2 &   2 & 0 \\
0 & -1 &   -3 & 4
\end{sysmatrix}
\\
&\!\begin{aligned}
&\ro{r_2+2r_3}\\
&\ro{-r_1}
\end{aligned}
\begin{sysmatrix}{rrrr}
1 &  -2 &   0 & -1 \\
0 &  1 & 3 & -4 \\
0 &  -2 & 2 & 0 \\
0 &  -1 & -3 & 4
\end{sysmatrix}
&\!\begin{aligned}
&\ro{r_3+2r_2}\\
&\ro{r_4+r_2}
\end{aligned}
\begin{sysmatrix}{rrrr}
1 &  -2 &   0 & 1 \\
0 &  1 & 3 & -4 \\
0 &  0 &   8 & -8 \\
0 &  0 & 0 & 0
\end{sysmatrix}
\\
&\!\begin{aligned}
&\ro{c_2+2c_1}\\
&\ro{c_4-c_1}
\end{aligned}
\begin{sysmatrix}{rrrr}
1 &  0 &   0 & 0 \\
0 &  1 & 3 & -4 \\
0 &  0 & 8 & -8 \\
0 &  0 & 0 & 0
\end{sysmatrix}
&\!\begin{aligned}
&\ro{c_3-3c_2}\\
&\ro{c_4+4c_2+c_3}
\end{aligned}
\begin{sysmatrix}{rrrr}
1 &  0 &   0 & 0 \\
0 &  1 & 0 & 0 \\
0 &  0 &   8 & 0 \\
0 &  0 & 0 & 0
\end{sysmatrix}
\end{alignat*}
\end{example}

It is often interesting to look at specific families of graphs and ask how to compute their critical groups.  As an example, the \emph{complete bipartite graph}\index{complete bipartite graph} $K_{m,n}$ has vertex set $\{x_1,\ldots, x_m, y_1, \ldots, y_n\}$ and edge set consisting of the edges between each $x_i$ and $y_j$ and no others.
\begin{exercise}
Find the critical group of $K_{3,3}$ by computing the Smith normal form of its Laplacian matrix.  How would your results generalize to other complete bipartite graphs $K_{m,n}$?
\end{exercise}
We note that a formula for the critical groups of all complete multipartite graphs is given in \cite{JNR}.

Many authors have worked on problems about computing critical groups for other special families of graphs.  For example, the critical groups of \emph{wheel graphs} are described in \cite[\S 9]{Biggs}, \emph{rook graphs} are considered in \cite{Rook}, and \emph{Paley graphs} in \cite{Paley}.  Much of this work has been done by undergraduate students, and there are many families of graphs that one could still explore!

We close this section by discussing one family of graphs where there are still many open questions about the critical groups, the \index{circulant graphs}circulant graphs.  To be explicit, the \emph{circulant graph} $C_n(a_1,\ldots,a_k)$ is formed by placing $n$ points on a circle and drawing the edges from each vertex to the vertices that are $a_1, a_2,\ldots,a_k$ positions further in the clockwise direction.  Two examples are given in Figure \ref{F:circulant}.  Some graphs of this form have been analyzed in \cite{GM} and \cite{HWC}, where results like the following are shown:

\begin{theorem}
Let $F_n$ be the $n^{th}$ Fibonacci number and let $d=\gcd(n,F_n)$.  Then the critical group of $C_n(1,2)$ is isomorphic to $\ZZ/d\ZZ \oplus \ZZ/F_n\ZZ \oplus \ZZ/(nF_n/d)\ZZ$.
\end{theorem}

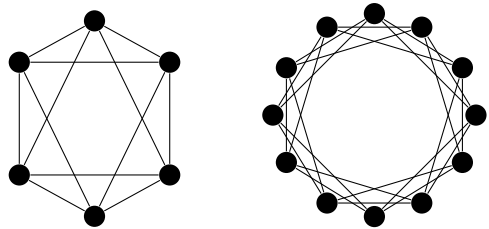
\begin{figure}
\begin{center}

\begin{tikzpicture} [scale=1,auto=left,every node/.style={circle,fill=black}]
\node (lu) at (-1,.75) {};
\node (ld) at (-1,-.75) {};
\node (u) at (0,1.3) {};
\node (ru) at (1,.75) {};
\node (rd) at (1,-.75) {};
\node (d) at (0,-1.3) {};
\foreach \from/\to in {lu/ld,lu/u,lu/ru,lu/d,ld/u,ld/rd,ld/d,u/ru,u/rd,ru/rd,ru/d,rd/d}
    \draw (\from) -- (\to);
\end{tikzpicture}
\hskip .4in
\begin{tikzpicture} [scale=.9,auto=left,every node/.style={circle,fill=black}]
\node (1) at (-1.3,-.7) {};
\node (2) at (-1.5,0){};
\node (3) at (-1.3,.7) {};
\node (4) at (-.7,1.3){};
\node (5) at (0,1.5) {};
\node (6) at (.7,1.3){};
\node (7) at (1.3,.7) {};
\node (8) at (1.5,0) {};
\node (9) at (1.3,-.7) {};
\node (10) at (.7,-1.3){};
\node (11) at (0,-1.5) {};
\node (12) at (-.7,-1.3){};
\foreach \from/\to in {1/3,2/4,3/5,4/6,5/7,6/8,7/9,8/10,9/11,10/12,11/1,12/2,1/4,2/5,3/6,4/7,5/8,6/9,7/10,8/11,9/12,10/1,11/2,12/3}
    \draw (\from) -- (\to);
\end{tikzpicture}

\caption{The circulant graphs $C_6(1,2)$ and $C_{12}(2,3)$}
\label{F:circulant}
\end{center}
\end{figure}

\begin{exercise}
Write down the Laplacian matrix for the graph $C_6(1,2)$. Verify that the critical group of this graph is $\ZZ/6\ZZ \oplus (\ZZ/13\ZZ)^2$.
\end{exercise}

An unpublished note \cite{circulant} argues that in general the critical group of the circulant graph $C_n(a,b)$ can be generated by at most $2b-1$ elements.  The authors also describe explicit calculations giving a library of the critical groups of all circulant graphs with at most $27$ vertices.

\resproject{Compute the critical groups of $C_n(1,3)$ for $n \le 10$, either by hand or using a computer algebra system.  Try to find patterns.}

\subsection{Elements of the Critical Group}\label{sec_elements}

In the previous section we saw how to determine the critical group of a graph by computing the Smith normal form of the Laplacian of $G$.  We also have seen how equivalence classes of divisors on a graph give elements of the critical group.
\begin{question}
How do we write down representatives for the elements of $K(G)$?  In particular, we know that $K(G)$ is isomorphic to the group of all classes of degree $0$ divisors on $G$ under chip-firing equivalence.  How do we make a `good choice' of one divisor from each class?  How do we determine if two divisors are in the same class?
\end{question}

There are several different approaches to choosing a representative from each class, and we will give one here.  Let $\delta \in \Div(G)$ and $v \in V(G)$.  We say that $v$ is \emph{in debt} if $\delta(v) < 0$.  We fix a vertex $q \in V(G)$ and define a divisor $\delta \in \Div(G)$ to be \emph{$q$-reduced} \index{$q$-reduced divisors} if $\delta(v) \ge 0$ for all $v \ne q$ and, moreover, for every nonempty set of vertices $A \subseteq V(G) \setminus \{q\}$, if one starts with the divisor $\delta$ and simultaneously fires every vertex in $A$ then some vertex in $A$ goes into debt.

\begin{example}
Once again we consider the graph from Example \ref{E:run}.  We denote the upper-left vertex as $q=v_1$, the lower-left as $v_2$, the upper-right as $v_3$, and the lower-right as $v_4$.  In order for a divisor $\delta$ to be $q$-reduced, one first notes that $\delta(v) < \deg(v)$ for all $v \ne q$ to account for the situation when $A$ is a single vertex.  On the other hand, if we fire all three of the vertices in $A=\{v_2,v_3,v_4\}$ then $\delta$ decreases by one at each of these vertices, so if firing at each vertex of $A$ causes one of the vertices to go into debt we know that the value of $\delta$ is zero for at least one of them.

Firing both vertices in $A=\{v_2,v_3\}$ decreases the value of the divisor at each of these vertices by two, which will already make both of the values negative by our above reasoning.  If $A=\{v_2,v_4\}$ then firing both vertices in $A$ decreases $\delta(v_2)$ by one and $\delta(v_4)$ by two.  In particular, if $\delta(v_4)=2$ then $\delta(v_2)=0$ and if $\delta(v_2)=1$ then $\delta(v_4)=0$ or $1$.  Considering $A=\{v_3,v_4\}$ gives the analogous results for $v_3$.

\begin{figure}
\begin{center}
\begin{tikzpicture}
  [scale=.3,auto=left,every node/.style={circle,fill=blue!20,minimum size=.7cm}]
  \node (l1) at (1,7) {$-2$};
  \node (l2) at (1,1) {$1$};
  \node (r1) at (7,7)  {$1$};
  \node (r2) at (7,1)  {$0$};
\foreach \from/\to in {l1/r1,l2/r2,l1/l2,r1/r2,l1/r2}
    \draw (\from) -- (\to);
\end{tikzpicture}
\hskip .3in
\begin{tikzpicture}
  [scale=.3,auto=left,every node/.style={circle,fill=blue!20,minimum size=.7cm}]
  \node (l1) at (1,7) {$-3$};
  \node (l2) at (1,1) {$1$};
  \node (r1) at (7,7)  {$1$};
  \node (r2) at (7,1)  {$1$};
\foreach \from/\to in {l1/r1,l2/r2,l1/l2,r1/r2,l1/r2}
    \draw (\from) -- (\to);
\end{tikzpicture}

\end{center}
\caption{Two divisors on the graph from Example \ref{E:run}.  The first is $q$-reduced, as one can see by firing each of the seven nonempty subsets of $\{v_1,v_2,v_3\}$.  The second is not $q$-reduced, as one can see by noting that firing all of the vertices in $\{v_1,v_2,v_3\}$ will not put any of these vertices into debt.}
\end{figure}
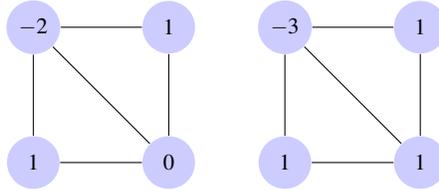

Combining these facts, one can see that there are eight $q$-reduced divisors of degree zero on this graph, given by the $4$-tuples $(\delta(v_1),\delta(v_2),\delta(v_3),\delta(v_4))$:
\begin{eqnarray*}
& \{ (0,0,0,0),(-1,1,0,0),& (-1,0,1,0),(-1,0,0,1), \\
& & (-2,1,1,0),(-2,1,0,1),(-2,0,1,1),(-2,0,2,0)\}.
\end{eqnarray*}

\end{example}

\begin{exercise}
Show that if we had chosen $q$ to be the vertex $v_2$ instead of $v_1$ that there would still be eight $q$-reduced divisors of degree zero.
\end{exercise}

As was suggested by the previous example, the number of $q$-reduced divisors does not depend on the choice of vertex $q$, even though the specific set of divisors certainly does. In fact, a much stronger result is true:

\begin{theorem}\cite[Prop 3.1]{BN1}
Let $G$ be a finite connected graph and $q \in V(G)$.  Then every divisor class in $K(G)$ contains a unique $q$-reduced divisor.
\end{theorem}

Checking whether or not a divisor is $q$-reduced directly from the above definition is difficult for large graphs as there are exponentially many subsets $A$ one needs to check.  However, there is a fast algorithm due to Dhar known as the \emph{Burning Algorithm}\index{Dhar's Burning Algorithm} that verifies whether a divisor is $q$-reduced by checking only a linear number of firing sets.  We will not give the details of this algorithm but refer the interested reader to \cite[Section 2.6.7]{Klivans}.  It is worth noting that $q$-reduced divisors were independently developed under the name of $G$-parking functions in order to generalize what are now called classical parking functions; for more details about this story, we refer the reader to \cite[Section 3.6]{Klivans}.

\subsection{Spanning Trees and the Matrix Tree Theorem}\label{sec:spanning}

A \index{spanning tree} \emph{spanning tree} of a connected graph $G$ is a subgraph $\CT$ consisting of all of the vertices of $G$ and a subset of the edges of $G$ so that the graph $\CT$ is connected and contains no cycles.  It follows from elementary results in graph theory that if $G$ (and hence $\CT$) has $n$ vertices then $\CT$ will have $n-1$ edges.

\begin{example}
Consider the cycle on $n$ vertices, $C_n$. We get a spanning tree by deleting any single edge.  Thus, $C_n$ has $n$ spanning trees.

The graph from Example \ref{E:run} consists of $4$ vertices and $5$ edges, so any spanning tree will be obtained by deleting two of the edges from the graph.  However, in this case we cannot just delete any two edges; for example, deleting the edges $\overline{v_1v_2}$ and $\overline{v_2v_4}$ will leave us with a graph that is both disconnected and contains a cycle (see Figure \ref{F:spantree}).  In particular, if we delete the edge $\overline{v_1v_4}$ then we can delete any of the remaining edges as our second edge.  Otherwise, we must delete exactly one edge from $\{\overline{v_1v_2},\overline{v_2v_3}\}$ and one from $\{\overline{v_1v_3},\overline{v_3v_4}\}$.  In particular, there are eight spanning trees of this graph.

\begin{figure}
\begin{center}
\begin{tikzpicture}
  [scale=.3,auto=left,every node/.style={circle,fill=blue!20}]
  \node (l1) at (1,7) {$v_1$};
  \node (l2) at (1,1) {$v_2$};
  \node (r1) at (7,7)  {$v_3$};
  \node (r2) at (7,1)  {$v_4$};
\foreach \from/\to in {l1/r1,r1/r2,l1/r2}
    \draw (\from) -- (\to);
\end{tikzpicture}
\hskip .3in
\begin{tikzpicture}
  [scale=.3,auto=left,every node/.style={circle,fill=blue!20}]
  \node (l1) at (1,7) {$v_1$};
  \node (l2) at (1,1) {$v_2$};
  \node (r1) at (7,7)  {$v_3$};
  \node (r2) at (7,1)  {$v_4$};
\foreach \from/\to in {l1/r1,l2/r2,l1/l2}
    \draw (\from) -- (\to);
\end{tikzpicture}
\hskip .3in
\begin{tikzpicture}
  [scale=.3,auto=left,every node/.style={circle,fill=blue!20}]
  \node (l1) at (1,7) {$v_1$};
  \node (l2) at (1,1) {$v_2$};
  \node (r1) at (7,7)  {$v_3$};
  \node (r2) at (7,1)  {$v_4$};
\foreach \from/\to in {l2/r2,r1/r2,l1/r2}
    \draw (\from) -- (\to);
\end{tikzpicture}
\caption{Three subgraphs of the graph from Example \ref{E:run} which each have three edges.  The first one is not a spanning tree but the other two are.}
\label{F:spantree}
\end{center}
\end{figure}
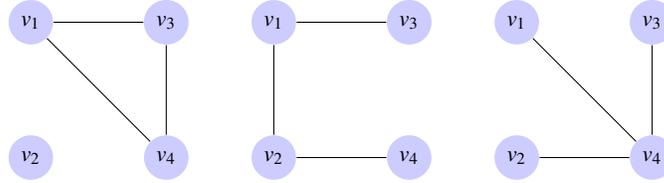
\end{example}

In general, it might appear to be a difficult question to ask for the number of spanning trees a given graph, but there is a nice answer given in terms of the Laplacian of the graph.  The result is often attributed to Kirchhoff based on work he did as an undergraduate in the 1840's. Many different proofs have been given over the years.  For a discussion of the history of this theorem as well as a proof and some related results, see \cite{Kirchhoff} and \cite[Chapter 6]{BiggsBook}.  Recall that the reduced Laplacian $L^{i,j}(G)$ of a graph $G$ is the matrix we get by deleting the $i^{th}$ row and $j^{th}$ column from $L(G)$.

\begin{theorem}[\index{Matrix Tree Theorem}Matrix Tree Theorem]\label{thm_matrix_tree}
The number of spanning trees of $G$ is equal to $|\det(L^{i,j}(G))|$ for any $i,j$.
\end{theorem}

Combining this theorem with the discussion in Section \ref{SS:SNF} gives us the following result which we will make use of repeatedly:

\begin{corollary}\label{C:Kirchhoff}
The order of the critical group $K(G)$ is the number of spanning trees of $G$.
\end{corollary}

In fact, Cori and LeBorgne give an explicit bijection between spanning trees of a graph and reduced divisors in the critical group in \cite{CL}.  In \cite{BS}, Baker and Shokrieh reformulate the question in terms of minimizing energy potential to generalize these results further.  We will not discuss these refinements here.

Corollary \ref{C:Kirchhoff} immediately tells us that any tree has trivial critical group, a fact that we will give a different proof of in Corollary \ref{C:tree}.  It also tells us that the critical group of a cycle on $n$ vertices has order $n$ and that the critical group of the graph in Example \ref{E:run} has order $8$, although it does not help us pin down the group exactly.  We will return to critical groups of cycles in the next section.

\begin{exercise}
Consider the `house graph' pictured here:
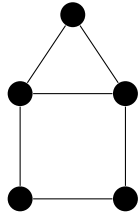
\begin{center}
\begin{tikzpicture}
[scale=.7,auto=left,every node/.style={circle,fill=black}]
\node (v1) at (0,0) {};
\node (v2) at (2,0) {};
\node (v3) at (2,2) {};
\node (v4) at (0,2) {};
\node (v5) at (1,3.5) {};
\draw (v1) -- (v2);
\draw (v3) -- (v2);
\draw (v3) -- (v4);
\draw (v4) -- (v1);
\draw (v3) -- (v5);
\draw (v4) -- (v5);
\end{tikzpicture}
\end{center}

\noindent Show that there are $11$ different spanning trees of this graph, and conclude that the critical group must be $\ZZ/11\ZZ$.  More generally, what can we say about the critical group of the graph consisting of two cycles sharing a common edge?
\end{exercise}

At the beginning of Section \ref{sec_elements} we noted that there are several approaches to choosing one divisor from each divisor class and then discussed the example of $q$-reduced divisors.  Another interesting choice comes from the theory of \emph{break divisors}, which are defined in terms of the spanning trees of $G$.  An, Baker, Kuperberg and Shokrieh use these divisors to give a decomposition of $\operatorname{Pic}^g(G)$, the set of all divisors of degree $d$ on $G$ modulo chip-firing equivalence \cite{ABKS}. This leads to a `geometric proof' of Theorem \ref{thm_matrix_tree}.

\subsection{How Does the Critical Group Change Under Graph Operations?}

To this point, we have used techniques from linear algebra to compute critical groups.  One can also often use combinatorial properties of graphs to help with these computations.   In this section, we will consider several such approaches.

\vspace{.5cm}

\noindent\textbf{The Dual of a Planar Graph}: A graph $G$ is \emph{planar} if it can be drawn on a sheet of paper without any edges crossing.  The \index{dual graph}\emph{dual graph} $\hat{G}$ is defined as follows.  Choose a drawing of $G$. The vertices of $\hat{G}$ are in bijection with the planar regions of the drawing.  There is an edge connecting two vertices of $\hat{G}$ precisely when the corresponding regions of the drawing of $G$ share an edge.  Two examples are given in Figure \ref{F:dual}.  We note that the dual of a planar simple graph may have multiple edges between two vertices.

\begin{figure}
\begin{center}

\begin{tikzpicture}[scale=.5,
V/.style = {
            draw,circle,thick,fill=#1},
V/.default = black!20
                        ]
\node[V] (A) at (0,0)[] {};
\node[V] (B) at (3,0) {};
\node[V] (D) at (0,-3) {};
\node[V] (E) at (3,-3) {};
    \begin{scope}[V/.default = black]
\node[V] (ABD) at (1,-1) {};
\node[V] (BDE) at (2,-2) {};
\node[V] (out) at (-2,2) {};

    \end{scope}

\draw
(A) to  (B)
(A) to  (D)
(D) to  (B)
(D) to  (E)
(E) to (B);

\clip (-4,-6) rectangle + (10.5,9.2);
\draw[dashed]
    (ABD) to  (BDE)
    (ABD) to [bend right=50,swap,near end] (out)
    (out) to [bend right=50]  (ABD)
    (out) to  [out=-135,in=-90,looseness=2] (BDE)
    (out) to  [out=45,in=0,looseness=2] (BDE);
    \end{tikzpicture}
\hskip .1in
\begin{tikzpicture}[ scale=.5,
V/.style = {
            draw,circle,thick,fill=#1},
V/.default = black!20
                        ]
\node[V] (A) at (0,0)[] {};
\node[V] (B) at (3,0) {};
\node[V] (D) at (0,3) {};
\node[V] (E) at (3,3) {};
    \begin{scope}[V/.default = black]
\node[V] (ABDE) at (1.5,1.5) {};
\node[V] (out) at (5,5) {};

    \end{scope}

\draw
(A) to  (B)
(A) to  (D)
(D) to  (E)
(E) to (B);

\clip (-3,-3) rectangle + (10,10);
\draw[dashed]
    (ABDE) to [out = 90, in = 180, looseness=1](out)
    (ABDE) to [out = -90, in = -90, looseness=2.7](out)
    (ABDE) to [out = 180, in = 180, looseness=2.7](out)
    (ABDE) to [out = 0, in = -90, looseness=1](out);

    \end{tikzpicture}

\end{center}

\caption{Two planar graphs and their duals.  The vertices of the original graphs are given in gray and the edges are solid.  The vertices of the dual graphs are given in black and the edges are dashed.}
\label{F:dual}
\end{figure}
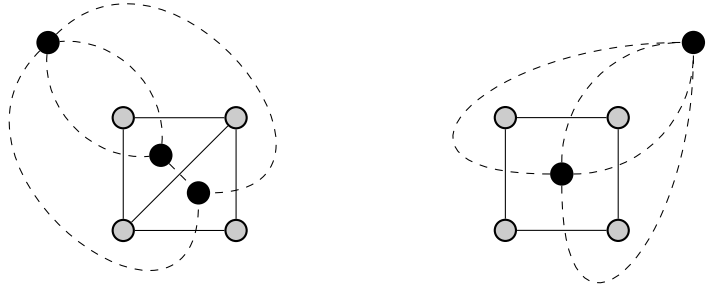

This definition of the dual depends on a choice of embedding into the plane. In particular there are graphs where different embeddings into the plane lead to non-isomorphic dual graphs.  That said, we have the following result due to Cori and Rossin \cite[Theorem 2]{CR}:

\begin{theorem}\label{T:dual}
If $G$ is a planar graph and $\hat{G}$ is its  dual graph then $K(G) \cong K(\hat{G})$.
\end{theorem}

\begin{corollary}\label{T:cycles}
The critical group of the cycle graph $C_n$ is $\ZZ/n\ZZ$.
\end{corollary}

\begin{proof}
The dual graph to $C_n$ consists of two vertices (one representing the inside of the cycle and one representing the outside) with $n$ edges between them, as illustrated in Figure \ref{F:dual}.  Therefore,
\[
L(\hat{C_n}) = \left(
\begin{array}{cc}
n & -n \\
-n & n
\end{array}
\right).
\]
We easily deduce that $K(\hat{C_n}) \cong \ZZ/n\ZZ$.  The result follows from Theorem \ref{T:dual}.
\end{proof}

In this argument we took the dual graph of a cycle and got a graph that had $n$ distinct edges between our pair of vertices.  As we noted earlier, standard facts about critical groups work in this more general \emph{multigraph} setting-- it is a good exercise to check that you believe us!

There is a construction similar to the dual graph known as the {\it line graph} of a graph $G$.  In particular, the line graph of $G$ is the graph $G_L$ whose vertices are in bijection with the edges of $G$ and two vertices in $G_L$ have an edge between them if and only if the corresponding edges share a vertex.  For information on critical groups of line graphs see \cite{BMMPR}.

\vspace{.5 cm}

\noindent\textbf{The Wedge of Two Graphs}: Let $G_1$ and $G_2$ be two finite graphs with designated vertices $v_1 \in G_1$ and $v_2 \in G_2$.  The \emph{wedge}\index{wedge of two graphs} of $G_1$ and $G_2$ is the graph $G$ consisting of the two graphs $G_1$ and $G_2$ with the vertices $v_1$ and $v_2$ identified.

\begin{example}

Let $G$ be the wedge of two triangles, as shown below.

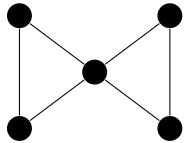
\begin{center}
\begin{tikzpicture} [scale=1,auto=left,every node/.style={circle,fill=black}]
\node (lu) at (-1,.75) {};
\node (ld) at (-1,-.75) {};
\node (c) at (0,0) {};
\node (ru) at (1,.75) {};
\node (rd) at (1,-.75) {};

\foreach \from/\to in {lu/ld,lu/c,ld/c,ru/c,rd/c,ru/rd}
    \draw (\from) -- (\to);
\end{tikzpicture}
\end{center}

\noindent One can check that
\[
L(G)
=\left(
      \begin{array}{ccccc}
        2 & -1 & -1 & 0 & 0 \\
        -1 & 2 & -1 & 0 & 0 \\
        -1 & -1 & 4 & -1 & -1 \\
        0 & 0 & -1 & 2 & -1 \\
        0 & 0 & -1 & -1 & 2 \\
      \end{array}
    \right).
    \]
Deleting the third row and third column of $L(G)$, gives block matrix consisting of two copies of the $2 \times 2$ matrix $\left(\begin{smallmatrix} 2 & -1 \\ -1 & 2\end{smallmatrix}\right)$.
It is straightforward to see that each of these blocks is the reduced Laplacian of a single triangle graph, and therefore the reduced Laplacian of the original graph can be reduced through row and column operations to
\[\left(
    \begin{array}{cccc}
      1 & 0 & 0 & 0 \\
      0 & 3 & 0 & 0 \\
      0 & 0 & 1 & 0 \\
      0 & 0 & 0 & 3 \\
    \end{array}
  \right).\]
By Corollary \ref{C:order}, the critical group of $G$ is $\ZZ/3\ZZ \oplus \ZZ/3\ZZ$.
\end{example}

This example generalizes, as shown in the following theorem:

\begin{theorem}\label{T:wedge}
Let $G_1$ and $G_2$ be two finite graphs and let $G$ be the wedge of $G_1$ and $G_2$.  Then $K(G) \cong K(G_1) \oplus K(G_2)$.
\end{theorem}

\begin{exercise}
Give a proof of Theorem \ref{T:wedge} in the spirit of the previous example.  In particular, if $G$ is the wedge of $G_1$ and $G_2$, determine the relationship between $L(G),\ L(G_1)$ and $L(G_2)$ and use this to compute the cokernel of $L(G)$ in terms of $\cok(L(G_1))$ and $\cok(L(G_2))$.
\end{exercise}

The following result follows immediately from Corollary \ref{C:Kirchhoff}, but we will give an additional proof illustrating the ideas of this section.

\begin{corollary}\label{C:tree}
Let $G$ be any tree.  Then the critical group $K(G)$ is trivial.
\end{corollary}

\begin{proof}
If $H$ is the graph consisting of two vertices and a single edge then $L(H) = \left(
                                                                                                                                   \begin{smallmatrix}
                                                                                                                                     1 & -1 \\
                                                                                                                                     -1 & 1
                                                                                                                                   \end{smallmatrix}
                                                                                                                                 \right)$. In particular it is clear that $K(H)$ is trivial.  Any tree can be constructed as the successive wedges of graphs isomorphic to $H$ and therefore the critical group of a tree is itself trivial.
\end{proof}

\vspace{.5cm}

\noindent\textbf{Adding/Subtracting an Edge}:
The fundamental theorem of finite abelian groups tells us that any finite abelian group $H$ can be written uniquely as a direct sum
\[
H \cong \ZZ/n_1 \ZZ \times \ZZ/n_2 \ZZ \times \cdots \times \ZZ/n_r \ZZ.
\]
where $n_{i} \mid n_{i+1}$ for all $i$ and $n_r > 1$.  The $n_i$ are the \emph{invariant factors}\index{invariant factors} of $H$, and the integer $r$ is the \emph{rank} of $H$, the minimum size of a generating set of $H$.  Let $G$ be a finite connected graph and $G'$ be a graph on the same set of vertices where we have added one additional edge.  Lorenzini shows that the rank of $K(G)$ and the rank of $K(G')$ differ by at most $1$ \cite[Lemma 5.3]{Lor89}.

Lorenzini uses this result to give an upper bound for the rank of the critical group of a connected graph $G$.  Since $K(G)$ is isomorphic to the cokernel of an $(n-1) \times (n-1)$ matrix, it is clear that the rank of $K(G)$ is at most $n-1$.  This bound is in general not good, and in fact we will see evidence in Section \ref{SS:random} that most graphs have cyclic critical groups.   Recall that the \emph{genus}\index{genus of a graph} of a graph is the number of independent cycles that the graph contains; in particular, it can be computed as $g(G)=|E(G)|-|V(G)|+1$.

\begin{theorem}\label{T:genus}\cite[Proposition 5.2]{Lor89}
Let $G$ be a connected graph and let $h(G)$ denote the rank of $K(G)$.  Then $h(G) \le g(G)$.
\end{theorem}

One can see that this bound is sharp by considering the graph formed as the wedge of $k$ copies of the triangle $C_3$.  This graph has genus $k$ and critical group $(\ZZ/3\ZZ)^k$.  In general, finding a minimal set of generators is an open problem.  We will return to this question in Section \ref{SS:generators}.

\vspace{.5cm}

\noindent\textbf{Subdividing an Edge}:  Let $G$ be a graph with $v_1, v_2 \in V(G)$ and $\overline{v_1 v_2} \in E(G)$.  Let $G'$ be the graph whose vertex set is the same as $G$ except with the edge $\overline{v_1 v_2}$ replaced with a path of $k$ edges.  We see that $V(G')$ consists of $V(G)$ together with $k-1$ new vertices along this path.

Subdividing a single edge of a graph can have all kinds of different effects on the critical group; If you subdivide an edge on a path then it does not change the critical group, as it will still be trivial, but if you subdivide an edge on the cycle $C_n$, replacing it with a path of length $2$, it changes the critical group from $\ZZ/n\ZZ$ to $\ZZ/(n+1)\ZZ$.  Subdividing an edge can change not only the order of $K(G)$, but can also change whether or not this group is cyclic, as illustrated in Figure \ref{F:subdivide}.

\begin{figure}[h]
\begin{center}
\subfloat[$K \cong \ZZ/4\ZZ \oplus \ZZ/4\ZZ$]{
\begin{tikzpicture} [scale=.7,auto=left,every node/.style={circle,fill=black}]
\node (ll) at (-2,0) {};
\node (lu) at (-1,1) {};
\node (ld) at (-1,-1) {};
\node (c) at (0,0) {};
\node (ru) at (1,1) {};
\node (rd) at (1,-1) {};
\node (rr) at (2,0) {};

\foreach \from/\to in {ll/lu,ll/ld,ld/c,lu/c,c/ru,c/rd,ru/rr,rd/rr}
    \draw (\from) -- (\to);

\end{tikzpicture}}
\hskip .1in
\subfloat[$K \cong \ZZ/20\ZZ$]{
\begin{tikzpicture}[scale=.7,auto=left,every node/.style={circle,draw=black,fill=black, minimum size=.2cm}]
 \node (ll) at (-2,0) {};
 \node (lu) at (-1,1) {};
  \node (ld) at (-1,-1) {};
   \node (c) at (0,0) {};
    \node (ru) at (1,1) {};
    \node (rur) at (1.5,.5) {};
     \node (rd) at (1,-1) {};
      \node (rr) at (2,0) {};
  \foreach \from/\to in {ll/lu,ll/ld,ld/c,lu/c,c/ru,c/rd,ru/rr,rd/rr}
    \draw (\from) -- (\to);
\end{tikzpicture}
}
\hskip .1in
\subfloat[$K \cong \ZZ/4\ZZ \oplus \ZZ/6\ZZ$]{
\begin{tikzpicture}[scale=.7,auto=left,every node/.style={circle,fill=black}]
 \node (ll) at (-2,0) {};
 \node (lu) at (-1,1) {};
  \node (ld) at (-1,-1) {};
   \node (c) at (0,0) {};
    \node (ru) at (1,1) {};
    \node (rur) at (1.5,.5) {};
     \node (rud) at (1.5,-.5) {};
     \node (rd) at (1,-1) {};
      \node (rr) at (2,0) {};
  \foreach \from/\to in {ll/lu,ll/ld,ld/c,lu/c,c/ru,c/rd,ru/rr,rd/rr}
    \draw (\from) -- (\to);
\end{tikzpicture}
}

\caption{Pictured above is (a) a graph with critical group $\ZZ/4\ZZ \oplus \ZZ/4\ZZ$ (b) a graph with critical group $\ZZ/4\ZZ \oplus \ZZ/5\ZZ \cong \ZZ/20\ZZ$ obtained by subdividing the previous graph, and (c) a graph with the noncyclic critical group $\ZZ/4\ZZ \oplus \ZZ/6\ZZ$ obtained by another subdivision.\label{F:subdivide}}
\end{center}
\end{figure}
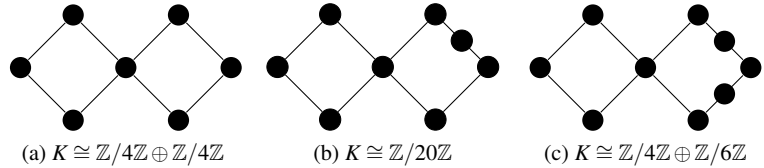

The following result from \cite{CY} shows that after a suitable choice of subdivisions one can always make the critical group cyclic.

\begin{theorem}\label{T:cyclicsub}
Let $G$ be a graph of genus $g \ge 1$.  Then there is a choice of at most $g-1$ subdivisions after which the critical group becomes cyclic.
\end{theorem}

\begin{exercise}
Show that Theorem \ref{T:cyclicsub} is true in the case where $G$ is the wedge of two cycle graphs $C_m$ and $C_n$.  In particular, this graph has genus two so you should show that either $K(G)$ is already cyclic or $K(G)$ can be made cyclic after a single subdivision.  Can you generalize this argument to the wedge of three or more cycles?
\end{exercise}

In a different vein, one can explicitly describe what happens after simultaneously subdividing all edges.  We begin with an example:
\begin{example}
Let $G$ be the graph consisting of the wedge of the cycles $C_3$ and $C_4$.  We have already seen that the critical group of $G$ is $K(G) \cong \ZZ/3\ZZ \oplus \ZZ/4\ZZ$.  Note that if we subdivide each edge of $G$ into $k$ edges then the new graph $G_k$ will be the wedge of the cycles $C_{3k}$ and $C_{4k}$ and therefore has critical group $K(G_k) \cong \ZZ/3k\ZZ \oplus \ZZ/4k\ZZ$.
\end{example}

It turns out that the previous example generalizes in a natural way.  Recall that Theorem \ref{T:genus} tells us that if $g$ is the genus of a graph $G$ then the critical group of $G$ can be written as $\ZZ/m_1\ZZ \oplus \ldots \oplus \ZZ/m_g\ZZ$, where it may be the case that some of the $m_i=1$.  We can use this decomposition to get the following result:

\begin{theorem}\cite[Proposition 2]{Lor91}\label{T:subdivide}
Let $G_{\text{sub}(k)}$ be the graph obtained by subdividing each edge of $G$ into $k$ edges.  Then, writing
\[
K(G) \cong \ZZ/m_1\ZZ \oplus \ldots \oplus \ZZ/m_g\ZZ
\]
as above we see that
\[
K(G_{\text{sub}(k)}) \cong\ZZ/km_1\ZZ \oplus \ldots \oplus \ZZ/km_g\ZZ.
\]
\end{theorem}

\begin{exercise}
Let $G$ be the graph from Example \ref{E:run} and let $G_{\text{sub}(2)}$ be the graph obtained by subdividing each edge of $G$ into two edges.  Compute the critical group of $G_{\text{sub}(2)}$ both by using Theorem \ref{T:subdivide} and by using results about the Laplacian matrix of $G_{\text{sub}(2)}$.
\end{exercise}

\vspace{.5cm}

\noindent\textbf{The Cone over a Graph}:   The \emph{join}\index{join of two graphs} of two graphs $G$ and $H$ consists of disjoint copies of $G$ and $H$ together with edges $\overline{uv}$ for all pairs $u \in V(G)$ and $v\in V(H)$.  The \emph{n\textsuperscript{th} cone over $G$}\index{cone over a graph}, denoted $G_n$, is the join of $G$ and the complete graph $K_n$.  Several authors have studied how the critical group of $G_n$ is related to the critical group of $G$ \cite{AlfaroValencia, BrownMorrowZB}.  The following result of Goel and Perkinson builds on these earlier efforts.  

\begin{theorem}\cite[Theorem 1]{GoelPerkinson}\label{T:cone}
Let $G$ be a connected graph on $k$ vertices, $n \ge 2$ be a positive integer, and $G_n$ be the $n$\textsuperscript{th} cone over $G$.  Let $\vone$ denote the $k \times k$ matrix whose entries are all $1$.
\begin{enumerate}
\item We have
\[
K(G_n) \cong \left(\ZZ/(n+k)\ZZ\right)^{n-2} \oplus \cok\left(n I_k + L(G) + \vone\right).
\]

\item The group $ \cok\left(n I_k + L(G) + \vone\right)$ has a subgroup isomorphic to $\ZZ/(n+k)\ZZ$.

\item We have
\[
|K(G_n)| = \frac{|p_{L(G)}(-n)|}{n} (n+k)^{n-1},
\]
where $p_{L(G)}$ is the characteristic polynomial of $L(G)$.
\end{enumerate}
\end{theorem}
The last of these statements is Corollary B in \cite{BrownMorrowZB}.

\begin{example}
Let $G$ be the path graph on two vertices.  One can see that $L(G) = \left(\begin{smallmatrix} 1 & -1 \\ -1 & 1\end{smallmatrix}\right)$, so that $p_{L(G)}(t)=t^2-2t$.  The third statement of this theorem therefore implies that $|K(G_n)|=(n+2)^n$ for all choices of $n$.  This does not tell us the specific group structure, although in this case we can see from the first statement that
\[
K(G_n) \cong \left(\ZZ/(n+2)\ZZ\right)^{n-2} \oplus \cok\left(\begin{smallmatrix} n+2 & 0 \\0 & n+2\end{smallmatrix}\right) \cong \left(\ZZ/(n+2)\ZZ\right)^{n}.
\]
\end{example}

When the graph is more complicated, Theorem \ref{T:cone} is more useful in determining the order of the critical group of the cone of a graph than in determining its group structure, something which \cite[Question 1.2]{BrownMorrowZB} asks about in a slightly different form.  Goel and Perkinson show that this involves understanding when $\ZZ/(n+k)\ZZ$ is a direct summand of $\cok(n I_k + L(G) + \vone)$.  This question is analyzed for the path on $4$ vertices in \cite[Example 5]{GoelPerkinson}.

\resproject{How much more can one say about the structure of $K(G_n)$ for a general graph $G$ and positive integer $n$, where $G_n$ is the $n$\textsuperscript{th} cone over $G$?}

\vspace{.5cm}

\noindent\textbf{Functions between Graphs}:   There are various results that look at the functorial properties of the critical groups of graphs.  One particularly nice example is given by \emph{Harmonic morphisms} between graphs, which Baker and Norine use to prove a graph-theoretic analogue of the Riemann-Hurwitz formula from algebraic geometry \cite{BN2}.  These morphisms induce different kinds of functorial maps between divisors on graphs and between their critical groups.  Reiner and Tseng examine the situation where one has a map between two graphs $\phi: G \rightarrow H$ that satisfies certain technical conditions and show that this induces a surjection of the critical groups $K(G) \twoheadrightarrow K(H)$ whose kernel can be understood  \cite{ReinerTseng}.  Other papers look at graphs that admit automorphisms and what one can say about either $|K(G)|$ or the structure of $K(G)$ in relation to its quotients.  For examples related to reflective symmetry see \cite{CYZ} and for dihedral group actions see \cite{GM}.

\subsection{Which Finite Abelian Groups Occur as the Critical Group of a Graph?}\label{sec_realiz}

Up to this point, we have primarily been concerned with the situation where we are given a graph $G$ and try to determine $K(G)$.  One could also ask how to construct graphs that have a given critical group.  Combining Theorem \ref{T:cycles} and Theorem \ref{T:wedge} implies that we can construct a graph with critical group
\[
\ZZ/m_1\ZZ \oplus \ldots \oplus \ZZ/m_d\ZZ
\]
by taking the wedge of cycles $C_{m_1}, C_{m_2},\ldots, C_{m_d}$.

\resproject{Let $H$ be a finite abelian group.  We know that there is \emph{some} graph $G$ with $K(G) \cong H$.  This $G$ is clearly far from unique.  What is the graph $G$ with the smallest number of vertices and given critical group?}
\noindent This is related to a problem of Rosa, which asks for the smallest number of vertices of a graph with a given number of spanning trees. Even this simpler sounding problem is not well understood.  See \cite{Sedlacek} for partial results.

There is a technical detail related to our discussion so far.  If any of the $m_i$ are equal to $2$, this construction taking a wedge of cycles $C_{m_i}$ does not result in a simple graph.   In fact, it is not difficult to show that there is no simple graph $G$ with $K(G) \cong \ZZ/2\ZZ$.  Suppose $G$ were such a graph and let $\CT$ be one of its spanning trees.  There must be some $e \in E(G)$ so that $\CT \cup \{e\}$ contains a cycle.  Since $G$ is a simple graph, this cycle has at least three edges.  Removing any edge in this cycle gives a spanning tree of $G$. Therefore, $G$ has at least three spanning trees, so $|K(G)| \ge 3$.  In \cite{GJRWW}, the authors significantly strengthen these ideas, and prove that there are no simple connected graphs with any of the following critical groups:
\[
\ZZ/2\ZZ \oplus \ZZ/4\ZZ, (\ZZ/2\ZZ)^2 \oplus \ZZ/4\ZZ, \ZZ/2\ZZ \oplus (\ZZ/4\ZZ)^2, \text{ or } (\ZZ/2\ZZ)^k \text{ for any } k \ge 1.
\]

\noindent Moreover, they show the following:
\begin{theorem}
Let $H$ be any finite abelian group.  There exists some positive integer $k_H$ so that there are no connected simple graphs with critical group $H \oplus (\ZZ/2\ZZ)^k$ for any $k \ge k_H$.
\end{theorem}

\resproject{Let $H \cong \ZZ/8\ZZ$.  For what values of $k$ is there a connected simple graph with critical group $\ZZ/8\ZZ \oplus (\ZZ/2\ZZ)^k$?

More generally, for other finite abelian groups $H$, what can we say about the value of $k_H$?  One approach to constructing such graphs might be to find graphs of a given genus and critical group and then subdividing each edge into two edges and using Theorem \ref{T:subdivide}.}

So far in this section we have asked only about the existence of a simple graph with a given critical group.  We can ask stronger questions about the existence of graphs with additional properties and given critical group.  For example, a graph $G$ has \emph{connectivity at least $\kappa$}\index{connectivity of graphs} if $G$ remains connected even if one deletes any set of $\kappa-1$ vertices and all edges incident to a vertex in this set.  In particular, a graph is said to be \emph{biconnected}\index{biconnected graph} if it remains connected after deleting any single vertex and all edges incident to it.  The authors of \cite{GJRWW} show that if a graph is biconnected and has maximum vertex degree $\delta$ then the critical group must contain some element whose order is at least $\delta$.  This result is one of the ingredients in proving that there are no simple graphs with critical group $(\ZZ/2\ZZ)^k$. These observations lead them to make the following conjecture.
\resproject{
Is it true that for any positive integer $n$, there exists $k_n$ such that if $k > k_n$, there is no biconnected graph $G$ with critical group $(\ZZ/n\ZZ)^k$?
}

\subsection{Generators of Critical Groups}\label{SS:generators}

In Section \ref{SS:random}, we will study properties of critical groups of random graphs and see that we often expect these critical groups to be cyclic.  The simplest possible nonzero divisor on $G$ is of the form $\delta_{xy}$ where $x,y \in V(G),\ \delta_{xy}(x)=1,\ \delta_{xy}(y)=-1$ and $\delta_{xy}(v)=0$ at all other vertices.
\begin{question}
Let $G$ be a connected finite graph with $K(G)$ cyclic. When does $K(G)$ have a generator of the form $\delta_{xy}$?
\end{question}
In \cite{BG}, the authors give a number of examples of graphs with cyclic critical groups and generators of this form, and also give examples of graphs with $K(G)$ cyclic that do not have a generator of this form.  They propose a general criterion for when a graph $G$ has such a generator.  This conjecture was proven in \cite{TwoVertex}.

\begin{theorem}\label{deltaxythm}
Let $x$ and $y$ be vertices on a finite connected graph $G$ and let $G'$ be the graph obtained by adding $\overline{xy}$ if $\overline{xy} \not\in E(G)$ and deleting $\overline{xy}$ if $\overline{xy} \in E(G)$.  Let $\delta_{xy}$ be defined as above and let $S \subseteq K(G)$ be the subgroup of the critical group of $G$ generated by $\delta_{xy}$.  Then we have the following relationships:
\begin{itemize}
\item $[K(G):S]$ divides $\gcd(|K(G)|,|K(G')|)$;
\item $\gcd(|K(G)|,|K(G')|)$ divides $[K(G):S]^2$.
\end{itemize}
In particular, $\delta_{xy}$ is a generator of $K(G)$ if and only if $\gcd(|K(G)|,|K(G')|)=1$.
\end{theorem}

\resproject{Theorem \ref{deltaxythm} gives a way of testing whether a given pair of vertices $x,y$ gives a divisor $\delta_{xy}$ that generates $K(G)$.  Is there a simple way to test whether there exists a pair of vertices $x,y$ such that $\delta_{xy}$ generates $K(G)$?

 For example, the wedge of a triangle, square, and pentagon has critical group $\ZZ/60\ZZ$, but there is no pair of vertices $x,y$ such that $\delta_{xy}$ generates $K(G$).}

\resproject{What happens when the critical group of $G$ is not cyclic?  For example, is there a way of testing whether two divisors $\delta_{x_1y_1}$ and $\delta_{x_2y_2}$ generate $K(G)$?}

\subsection{Critical Groups of Random Graphs} \label{SS:random}

In Section \ref{sec_realiz}, we saw that every finite abelian group occurs as the critical group of a graph if we allow multiple edges between vertices, and that every finite abelian group of odd order occurs as the critical group of a simple graph.  Instead of asking whether a group occurs as $K(G)$ for at least one graph $G$, we could ask about which kinds of groups occur often as the critical group of a graph.  Throughout this section we restrict our attention to simple graphs.
\begin{question}\index{random graph}
What can we say about critical groups in families of `random graphs'?
\end{question}

Here is one way to make this question precise.  There are $\binom{n}{2}$ possible edges between vertices $v_1,\ldots v_n$, so there are $2^{\binom{n}{2}}$ labeled simple graphs on this vertex set.  As a warmup, we can ask the following.
\begin{question}
How many of these $2^{\binom{n}{2}}$ graphs are connected and have trivial critical group?
\end{question}

Corollary \ref{C:Kirchhoff} implies that a connected graph has trivial critical group if and only if it is a tree.  It follows from Example \ref{ExampleKn} that the number of labeled trees of $n$ vertices is $n^{n-2}$.  So the proportion of graphs on $n$ vertices that are connected and have trivial critical group is $n^{n-2}/2^{\binom{n}{2}}$, which goes to zero as $n$ goes to infinity.  This tells us that the size of $K(G)$ is not often equal to $1$, but does not tell us how large we should expect it to be.

In order to determine the average size of the critical group of a graph on $n$ vertices, we introduce some ideas from probabilistic combinatorics.  There are $n^{n-2}$ trees on $n$ vertices, and each tree has exactly $n-1$ edges.  Fix a choice of a spanning tree $\CT$ on $n$ vertices.  The number of graphs on $n$ vertices containing $\CT$  as a subgraph will be $2^{\binom{n}{2}-(n-1)}$ since, for each edge not in $\CT$, we can choose whether it is present in our graph.  This implies that the probability that $\CT$ is contained in a random graph is $1/2^{n-1}$. It then follows from linearity of expectation that the expected number of spanning trees of a graph on $n$ vertices is $n^{n-2}/2^{n-1}$.  It is easy to check this formula in small cases.
\begin{example}
There are $8$ graphs with vertex set $\{v_1, v_2, v_3\}$, and $4$ of these are connected: the complete graph $K_3$, which has $3$ spanning trees, and $3^{3-2} = 3$ trees, which have $1$ spanning tree each.  We conclude that the average number of spanning trees of a graph on $3$ vertices is $3/4$.
\end{example}

A graph $G$ on $n$ vertices is not connected if and only if it does not contain any of the $n^{n-2}$ spanning trees of the complete graph with vertex set $V(G)$.
\begin{exercise}\label{connected_exer}
Show that as $n$ goes to infinity, the proportion of graphs on $n$ vertices that are connected goes to $1$.

Here is one approach: A graph $G$ with $n$ vertices is connected if every one of the $\binom{n}{2}$ pairs of vertices $v_i, v_j \in V(G)$ share a common neighbor.  What is the probability that $v_k$ is a common neighbor of both $v_i$ and $v_j$?  What is the probability that $v_i$ and $v_j$ do not share a common neighbor?
\end{exercise}
For the rest of this section, when we ask about the proportion of graphs $G$ on $n$ vertices for which $K(G)$ satisfies some property, what we really mean is the proportion of graphs $G$ that are connected and such that $K(G)$ has this property.  By Exercise \ref{connected_exer}, as $n$ goes to infinity the proportion of connected graphs goes to $1$, so we do not need to keep writing this extra assumption.

Since $n^{n-2}/2^{n-1}$ goes to infinity with $n$, we see that the average size of $K(G)$ gets large as $|V(G)|$ gets large.  In fact, something stronger is true:
\begin{proposition}\label{KGsmall}
Let $X$ be a positive integer.  The proportion of graphs $G$ on $n$ vertices for which $|K(G)| \le X$ goes to $0$ as $n$ goes to infinity.
\end{proposition}
Note that if $G$ has at most $X$ spanning trees, then we can make $G$ disconnected by removing at most $X$ edges, so $X$ has \emph{edge connectivity} at most $X$.  We leave the proof of this proposition as an exercise, but refer the interested reader to \cite[Chapter 4]{FriezeKaronski} for results on connectivity of random graphs.

A consequence of Proposition \ref{KGsmall} is that for any particular finite abelian group $H$, the probability that $K(G) \cong H$ goes to $0$ as $|V(G)|$ goes to infinity.  Instead of asking for $K(G)$ to be isomorphic to a particular group, we can ask for the probability that this group has some chosen property.

\begin{question}
What proportion of the $2^{\binom{n}{2}}$ graphs on $n$ vertices have $K(G)$ cyclic?
\end{question}
This  question has been the subject of much recent research including work of Wagner \cite{Wag00}, Lorenzini \cite{Lor08}, and Wood \cite{Woo17}.   One nice thing about this type of question is that it is not so difficult to do large experiments using a computer algebra system, for example Sage, and to get a sense for what to expect.  Building on work of \cite{CLP}, the authors of \cite{CKLPW} make the following conjecture.

\begin{conjecture}
We have
\begin{eqnarray*}
& & \lim_{n \to \infty} \frac{\#\{\text{Connected graphs } G \text{ with } |V(G)| = n\ \  \text{ and }\ \  K(G) \text{ cyclic}\}}{2^{\binom{n}{2}}}  \\
& & \ \ \ \ \ \ \ \ \ \ \ \ \ \ \ \ \ \ \ \ \ =
\zeta(3)^{-1} \zeta(5)^{-1} \zeta(7)^{-1} \zeta(9)^{-1} \zeta(11)^{-1} \cdots \approx .7935212.
\end{eqnarray*}
\end{conjecture}

In this conjecture, $\zeta(s) = \sum_{n=1}^\infty n^{-s}$ denotes the Riemann zeta function. Wood has proven that this conjectured value is an upper-bound for the probability that the critical group of a random graph is cyclic \cite[Corollary 9.5]{Woo17}. Showing that equality holds appears to be quite difficult.

It is also interesting to ask questions about other properties of the order of the critical group, such as the following:
\begin{question}\label{question_odd}
What proportion of the $2^{\binom{n}{2}}$ graphs on $n$ vertices have $|K(G)|$ odd?
\end{question}
That is, we would like to understand the following limit:
\begin{equation}\label{even_spanning}
 \lim_{n \to \infty} \frac{\#\{\text{Connected graphs } G \text{ with } |V(G)| = n\ \  \text{ and } \ \ |K(G)| \text{ odd}\}}{2^{\binom{n}{2}}}.
\end{equation}

One of the main ideas that goes into the study of these questions is that a finite abelian group $H$ decomposes as a direct sum of its Sylow $p$-subgroups.  Recall that the \emph{Sylow $p$-subgroup}\index{Sylow $p$-subgroup} of a finite abelian group $H$ is the subgroup of all of its elements of $p$-power order.  We denote this subgroup by $H_p$.  We can interpret many questions about $K(G)$ in terms of the Sylow $p$-subgroups $K(G)_p$.  For example, a connected graph $G$ has a cyclic critical group if and only if $K(G)_p$ is cyclic for each prime $p$.  Similarly, $G$ has an odd number of spanning trees if and only if $K(G)_2$ is trivial.  This suggests that a good starting place is to try to understand how the Sylow $p$-subgroups of critical groups of random graphs behave. The following result of Wood answers this question.
\begin{theorem}\cite[Theorem 1.1]{Woo17}\label{thm_wood}
Let $p$ be a prime and $H$ a finite abelian $p$-group.  Then
\begin{eqnarray*}
& &  \lim_{n \to \infty} \frac{\#\{\text{Connected graphs } G \text{ with } |V(G)| = n\ \  \text{ and }\ \  K(G)_p \cong H\}}{2^{\binom{n}{2}}} \\
& = & \frac{\#\{\text{symmetric, bilinear, perfect pairings  } \phi \colon H \times H \to \C^*\}}{|H| |\Aut(H)|} \prod_{k \ge 0} (1-p^{-2k-1}).
\end{eqnarray*}
\end{theorem}

We will discuss pairings on finite abelian $p$-groups and this theorem in more detail in Section \ref{sec_pairing}.  In the meantime, taking $p=2$ and $H$ equal to the trivial group, we see that the probability that a random graph has an odd number of spanning trees is $\prod_{k \ge 0}^{\infty} (1-2^{-2i-1}) \approx .4194$, answering Question \ref{question_odd}.

\vspace{.5cm}

\noindent\textbf{Critical groups of random graphs and cokernels of random integer matrices}:  Questions about critical groups of random graphs are closely connected to questions about random symmetric integer matrices.  When $R$ is equal to either $\ZZ$ or $\ZZ/p\ZZ$, we let $\Sym_n(R)$ denote the set of $n \times n$ symmetric matrices with entries in $R$.  To see the connection between random graphs and matrices, we note that half of the $2^{\binom{n}{2}}$ graphs $G$ with $V(G) = \{v_1,\ldots, v_n\}$ have $\overline{v_i v_j} \in E(G)$.  So choosing one of these $2^{\binom{n}{2}}$ graphs uniformly at random is the same as flipping a coin for each of the $\binom{n}{2}$ potential edges of the graph to decide whether to include it.  This implies that choosing a random graph on $n$ vertices and computing its critical group is the same as the following process:
\begin{enumerate}
\item Choose a random matrix $A \in \Sym_n(\ZZ)$ with all diagonal entries equal to $0$ by taking each pair $1 \le i < j \le n$ and setting $a_{i,j} = 0$ with probability $1/2$ and $a_{i,j} = 1$ with probability $1/2$.
\item Compute the diagonal matrix $D$ with $(i,i)$-entry equal to the negative of the sum of the entries in the $i$\textsuperscript{th} row of $A$.  Let $L_0$ be the $(n-1) \times (n-1)$ matrix that we get by deleting the last row and column of $D-A$.
\item Take the cokernel of $L_0$.
\end{enumerate}
Many questions about properties of random graphs can be phrased as questions about this family of random integer matrices.  For example, we have seen that a graph $G$ is connected if and only if $L_0(G)$ has rank $n-1$, so the proportion of graphs with $n$ vertices that are connected is the same as the probability that a random matrix $L_0$ chosen by the procedure above has rank $n-1$.

We will use the fact that $K(G)_p$ only depends on the entries of $L_0(G)$ modulo powers of $p$.
\begin{exercise}\label{reducemodp_exer}
Let $G$ be a connected graph.
\begin{enumerate}[(a)]
\item Prove that $K(G)_p$ is trivial if and only if $p \nmid \det(L_0(G))$.
\item Conclude that $K(G)_p$ is trivial if and only if we reduce the entries of $L_0(G)$ modulo $p$ and get a matrix with entries in $\ZZ/p\ZZ$ of rank $n-1$.
\end{enumerate}
\end{exercise}
How often should we expect $K(G)_p$ to be trivial?  Exercise \ref{reducemodp_exer} suggests that a good first step is to compute the proportion of all matrices in $\Sym_{n-1}(\ZZ/p\ZZ)$ that have rank $n-1$.

\begin{theorem}\cite[Theorem 2]{MacW}\label{MacWthm}
The number of invertible matrices in $\Sym_{n-1}(\ZZ/p\ZZ)$ is
\[
p^{\binom{n}{2}} \prod_{j=1}^{\lceil\frac{n-1}{2}\rceil} (1-p^{1-2j}).
\]
\end{theorem}
We leave the proof as a nice exercise in linear algebra over finite fields.

As we take $n$ to infinity, Theorem \ref{MacWthm} implies that the proportion of invertible matrices in $\Sym_{n-1}(\ZZ/p\ZZ)$ approaches $\prod_{k \ge 0}^{\infty} (1-p^{-2i-1})$.  This is the same probability that we get by taking the trivial group in Theorem \ref{thm_wood}, the probability that the number of spanning trees of a large random graph is not divisible by $p$.  Wood's theorem demonstrates a deep type of \emph{universality for cokernels of random matrices}\index{universality}.  Even though the reduced Laplacian of a random graph \emph{does not} give a uniformly random element of $\Sym_{n-1}(\ZZ/p\ZZ)$, as $n$ goes to infinity the probability that the reduced Laplacian modulo $p$ is an invertible matrix is the same as the proportion of matrices in $\Sym_{n-1}(\ZZ/p\ZZ)$ that are invertible.

In order to understand the Sylow $p$-subgroup of $\cok(L_0(G))_p$, we must consider not only the entries of $L_0(G)$ modulo $p$, but also modulo higher powers of $p$.  There is a nice algebraic setting for these questions.  Instead of thinking about $L_0(G)$ as a matrix with integer entries, we think of it as a matrix with entries in the \emph{$p$-adic integers}\index{$p$-adic integers}, which we denote by $\ZZ_p$.  A $p$-adic integer consists of an element of $\ZZ/p^k\ZZ$ for each $k$ that is compatible with the canonical surjections $\ZZ/p^k\ZZ \twoheadrightarrow \ZZ/p^{k-1}\ZZ$.  For any prime $p,\ \ZZ \subset \ZZ_p$ since the integer $n$ corresponds to choosing the residue class $n \pmod{p^k}$ for each $k$.  There is a nice description of how to choose a random matrix with $p$-adic entries that comes from the existence of Haar measure for $\ZZ_p$.  We do not give details here.  For an accessible introduction to $p$-adic numbers, we recommend Gouvea's book \cite{Gouvea}.

Clancy, Leake, and Payne performed large computational experiments about critical groups of random graphs and made conjectures based on their data \cite{CLP}.  Motivated by these conjectures, these authors together with Kaplan and Wood determine the distribution of cokernels of random elements of $\Sym_n(\ZZ_p)$  as $n$ goes to infinity \cite{CKLPW}.  Theorem \ref{thm_wood} is a consequence of a much stronger result of Wood about cokernels of families of random $p$-adic matrices \cite{Woo17}.  Wood proves that for a large class of distributions on the entries of such a matrix the distribution of the cokernels does not change.  This class is large enough to include reduced Laplacians of random graphs, so even though these matrices are very far from being uniformly random modulo powers of $p$, the distribution of their cokernels matches the distribution in the uniformly random setting.

\vspace{.5cm}

\noindent\textbf{Choosing a random graph}:  So far in this section we have chosen a random graph by choosing one of the $2^{\binom{n}{2}}$ graphs on $n$ vertices uniformly at random.  It is common in the study of random graphs to allow the probability of choosing a particular graph to be weighted by its number of edges.  Let $0 < q < 1$.  An \emph{Erd\H{o}s-R\'enyi random graph}\index{Erd\H{o}s-R\'enyi random graph} on $n$ vertices, $G(n,q)$, is a graph on $n$ vertices $v_1,\ldots, v_n$ where we independently include the edge $\overline{v_i v_j}$ with probability $q$.  That is, $G(n,q)$ is a probability space on graphs with $n$ vertices in which a graph with $m$ edges is chosen is with probability
\[
q^m (1-q)^{\binom{n}{2}-m}.
\]
We see that our earlier model of choosing a random graph corresponds to $G(n,1/2)$, in which each graph is chosen with equal probability.

The conjectures in \cite{CKLPW, CLP}, and the results of \cite{Woo17} apply in this more general Erd\H{o}s-R\'enyi random graph setting.  That is, if we choose an Erd\H{o}s-R\'enyi random graph $G$ on $n$ vertices with edge probability equal to some fixed constant $q$ (for example, $1/2$, or $2/3$, or $10^{-100}$), as $n$ goes to infinity the probability that $K(G)_p$ is isomorphic to a particular finite abelian $p$-group $H$ is given by the right-hand side of Theorem \ref{thm_wood}, no matter what value of $q$ we choose.  Again, this is a consequence of Wood's universality results for cokernels of random matrices \cite{Woo17}.

An active area of current research involves allowing the edge probability $q$ to change with $n$.  Linearity of expectation implies that the expected number of edges of a random graph $G(n,q)$ is $\binom{n}{2} q$.  Therefore, if we allow $q$ to go to $0$ as $n$ goes to infinity, but not too fast, this random graph will still have an increasing number of edges.
\begin{exercise}
Show that the probability that an Erd\H{o}s-R\'enyi random graph $G(n,n^{-1/2})$ is connected goes to $1$ as $n$ goes to infinity, even though $n^{-1/2}$ goes to $0$.
\end{exercise}
This exercise is more challenging than Exercise \ref{connected_exer}.  We again refer the interested reader to \cite[Chapter 4]{FriezeKaronski}.

It is likely that a version of Theorem \ref{thm_wood} holds when $q$ is allowed to go to $0$ or $1$ as $n$ goes to infinity, as long as it does not approach $0$ or $1$ too fast.  Determining the threshold where the behavior of the critical group changes is an interesting, and likely very challenging, open problem.  For work in this direction see the recent paper of Nguyen and Wood \cite{NguyenWood}.

\begin{question}\label{question_random}
What can we say about Sylow $p$-subgroups of critical groups in other families of random graphs?
\end{question}

We give two concrete examples to show what Question \ref{question_random} is all about.  A graph $G$ is \emph{bipartite}\index{bipartite graph} if we can divide its vertex set $V(G)$ into disjoint sets $V_1$ and $V_2$ so that every edge in $G$ connects a vertex in $V_1$ to a vertex in $V_2$.  We can choose a random bipartite graph with vertex set $V(G) = V_1 \cup V_2$ as follows.  Fix $0< q< 1$.  Independently include each of the $|V_1| |V_2|$ possible edges between a vertex in $V_1$ and a vertex in $V_2$ with probability $q$.
\resproject{Consider a random bipartite graph with edge probability $q$ and $|V_1| = |V_2| = n$.  As $n$ goes to infinity, how are the Sylow $p$-subgroups of the critical groups of these graphs distributed?}
Koplewitz shows that if the sizes of the vertex sets $V_1$ and $V_2$, are too `unbalanced', that is $|V_1|/|V_2| < 1/p$, then the resulting distribution of Sylow $p$-subgroups of the critical groups of these random bipartite graphs does not match the distribution given in Theorem \ref{thm_wood} \cite{Koplewitz}.

To give a second example, a graph $G$ is \emph{$d$-regular}\index{regular graphs} if every $v \in V(G)$ has degree $d$.  Fix a positive integer $d \ge 3$.  Choose a $d$-regular graph on $n$ vertices uniformly at random.  M\'esz\'aros has recently shown that as $n$ goes to infinity, the distribution of Sylow $p$-subgroups of critical groups of random $d$-regular graphs is the same as the one given by Theorem \ref{thm_wood}, except when $p=2$ and $d$ is even, in which case we get a different distribution \cite{Mezaros}.

These are just two examples of a large family of problems to investigate.
\resproject{Choose your favorite graph property $P$.  Is it true that the distribution of Sylow $p$-subgroups of large random graphs with property $P$ matches the distribution of Sylow $p$-subgroups of all random graphs?  For example, what is the distribution of Sylow $p$-subgroups of large random planar graphs?  What about random triangle-free graphs?}

\subsection{The Monodromy Pairing on Divisors}\label{sec_pairing}

The expression on the right side of Theorem \ref{thm_wood} contains a term that involves the number of symmetric, bilinear, perfect pairings on a finite abelian group $H$.  This is because the critical group of a graph comes with extra algebraic structure.  In order to explain this structure, we introduce some additional material about divisors on graphs closely following Shokrieh's presentation in \cite{Shokrieh}.

Our goal is to show how the critical group of a connected graph comes equipped with a symmetric, bilinear, perfect pairing.  We first show that the group of degree zero divisors on $G$ comes with a pairing, that is, a function $\langle \cdot, \cdot\rangle \colon \Div^0(G) \times \Div^0(G) \to \QQ$, and then, that this pairing descends to a pairing defined on $K(G)$.  Much of the following terminology for divisors on graphs is motivated by the analogy with divisors on algebraic curves that we first mentioned in Remark \ref{remark_AG}.

Recall that a divisor on a graph $G$ is a function $\delta \colon V(G) \to \ZZ$.  Let $\mathcal{M}(G)$ denote the abelian group consisting of integer-valued functions defined on $V(G)$, that is, $\mathcal{M}(G) = \Hom(V(G),\ZZ)$.  Let $f \in \mathcal{M}(G)$.  For $v\in V(G)$, we define
\[
\ord_v(f) = \sum_{\substack{w \in V(G) \\ \overline{vw} \in E(G)}} \left(f(v) - f(w)\right).
\]
The divisor of the function $f$, denoted $\divi(f)$, is defined by setting $(\divi(f))(v) = \ord_v(f)$ for any $v \in V(G)$.  Every $\divi(f)$ has degree $0$, but not every degree $0$ divisor is the divisor of a function $f$.   We say that a divisor is \emph{principal}\index{principal divisors} if it is equal to $\divi(f)$ for some $f\in \mathcal{M}(G)$ and denote the group of principal divisors on $G$ by $\Prin(G)$.

\begin{example}
Consider the graph consisting of a cycle on three vertices $\{u,v,w\}$.  For any function $f \in \mathcal{M}(G)$ we see that $\ord_u(f)=2f(u)-f(v)-f(w)$, $\ord_v(f)=2f(v)-f(u)-f(w)$, and $\ord_w(f)=2f(w)-f(v)-f(u)$.  It is clear that these three numbers sum to zero for any choice of $f$.  On the other hand, if we set $\delta$ to be the divisor of degree zero with $\delta(u)=0, \delta(v)=1, \delta(w)=-1$ then in order for $\delta$ to be principal there would have to be an integer-valued function so that
\begin{eqnarray*}
2f(u)-f(v)-f(w)&=&0\\
2f(v)-f(u)-f(w)&=&1\\
2f(w)-f(v)-f(u)&=&-1
\end{eqnarray*}
It is a simple exercise in linear algebra to see that this cannot happen.
\end{example}

\begin{exercise}
For the cycle from the previous example, describe which divisors of degree zero are principal and which are not.
\end{exercise}

\begin{exercise}
More generally, let $G$ be any connected graph.  If we identify $\Div(G)$ with column vectors of length $|V(G)|$ that have integer entries, we have seen that a divisor $D$ is chip-firing equivalent to the all zero divisor if and only if it is in the image of $L(G)$.  Show that $D$ is chip-firing equivalent to the all zero divisor if and only if it is principal.  Use this characterization to see that $K(G) \cong \Div^0(G)/\Prin(G)$.
\end{exercise}

We now describe the \emph{monodromy pairing}\index{monodromy pairing} on divisors on the critical group of a connected graph $G$, which is a graph-theoretic analogue of a notion called the \emph{Weil pairing}\index{Weil pairing} on the Jacobian of an algebraic curve.  Let $D_1, D_2 \in \Div^0(G)$ and let $m_1, m_2$ be integers such that $m_1 D_1$ and $m_2 D_2$ are principal.  (Such integers must exist because $K(G)$ is finite.) In particular, there will be functions $f_1, f_2 \in \mathcal{M}(G)$ such that $m_1 D_1 = \divi(f_1)$ and $m_2 D_2 = \divi(f_2)$.

\begin{exercise}
Show that
\[
\frac{1}{m_2} \sum_{v \in V(G)} D_1(v) f_2(v) =
\frac{1}{m_1} \sum_{v \in V(G)} D_2(v) f_1(v).
\]
\end{exercise}

We define a pairing $\langle \cdot, \cdot\rangle \colon \Div^0(G) \times \Div^0(G) \to \QQ$ by
\[
\langle D_1, D_2\rangle = \frac{1}{m_2} \sum_{v \in V(G)} D_1(v) f_2(v).
\]
By the previous exercise, $\langle D_1, D_2\rangle = \langle D_2, D_1 \rangle$ for all $D_1, D_2 \in \Div(G)$, that is, this pairing is \emph{symmetric}.  It is also not difficult to check that it is \emph{bilinear}\index{bilinear map}, meaning that $\langle aD_1+bD_2, D_3\rangle = a\langle D_1, D_3 \rangle + b \langle D_2, D_3 \rangle$ for all divisors $D_1,D_2,D_3$ and all rational numbers $a,b$.

A symmetric bilinear pairing on a finite abelian group $H$ is \emph{non-degenerate} if the group homomorphism defined by $h \to \langle h,\cdot\rangle$ is injective.  If it is an isomorphism, the pairing is called \emph{perfect}.  We write $\overline{D}$ for an element of $K(G)$ if $\overline{D}$ is the divisor class of $D$ in $K(G)$.  The following theorem states that the pairing on $\Div^0(G)$ descends to a well-defined perfect pairing on $K(G)$.
\begin{theorem}\cite[Theorem 3.4]{Shokrieh}\label{thm_pairing}
The pairing $\langle \cdot, \cdot\rangle \colon K(G) \times K(G) \to \QQ/\ZZ$ defined by
\[
\langle \overline{D_1}, \overline{D_2} \rangle = \frac{1}{m_2} \sum_{v \in V(G)} D_1(v) f_2(v) \pmod{\ZZ},
\]
where $m_2 D_2 = \divi(f_2)$ is a well-defined, perfect pairing on $K(G)$.
\end{theorem}
This pairing is called the \emph{monodromy pairing}\index{monodromy pairing}.  Shokrieh gives a concrete proof of Theorem \ref{thm_pairing} in Appendix A of \cite{Shokrieh}, but notes that the same result is proven in slightly different language by Bosch and Lorenzini in \cite{BoschLorenzini}.

The same underlying finite abelian group may have different perfect pairings defined on it.  Let $G$ be a finite abelian group and $\langle \cdot,\cdot\rangle_1$ and $\langle \cdot,\cdot\rangle_2$ be two pairings defined on $G$.  We say that these pairings are isomorphic if there exists $\varphi \in \Aut(G)$ such that for all $x,y\in G,\ \langle x,y\rangle_1 = \langle \varphi(x),\varphi(y) \rangle_2$. The following exercise contains some of the basics of the classification of pairings on finite abelian groups.  For much more on this topic see \cite{Miranda, Wall}.
\begin{exercise}
Let $p$ be an odd prime and $r$ be a positive integer.
\begin{enumerate}[(a)]
\item Show that every non-degenerate pairing $\langle \cdot, \cdot\rangle \colon \ZZ/p^r\ZZ \times \ZZ/p^r\ZZ \to \QQ/\ZZ$ is of the form
\[
\langle x,y\rangle_a = \frac{a xy}{p^r}
\]
for some integer $a$ not divisible by $p$.

\item Show that $\langle x,y\rangle_a$ is isomorphic to $\langle x,y\rangle_b$ if and only if the Legendre symbols $\legen{a}{p}$ and $\legen{b}{p}$ are equal.

\item Show that every finite abelian $p$-group with a perfect pairing decomposes as an orthogonal direct sum of cyclic groups with pairings.

\end{enumerate}

\end{exercise}

 Like many things in algebra, the prime $p=2$ behaves in a special way.  The classification of perfect pairings on finite abelian $2$-groups is significantly more complicated than in the case where $p$ is odd.  See \cite[Section 2.4]{GJRWW} for a discussion of these issues.  For any finite abelian group $H$, this material can be used to compute the term $\#\{\text{symmetric, bilinear, perfect pairings  } \phi \colon H \times H \to \C^*\}$ from Theorem \ref{thm_wood}; see equation (2) of \cite[p. 916]{Woo17}.

We can now revisit the material from each of the previous two sections and ask not only about finite abelian groups that occur as the critical group of a graph, but about finite abelian groups with a chosen perfect pairing.  In \cite{GJRWW}, the authors use a construction based on \emph{subdivided banana graphs} to show that odd order groups with pairings occur as critical groups.

\begin{theorem}\cite[Theorem 2]{GJRWW}
Assume the \emph{Generalized Riemann Hypothesis}.  Let $\Gamma$ be a finite abelian group of odd order with a perfect pairing on $\Gamma$.  Then there exists a graph $G$ such that $K(G) \cong \Gamma$ as groups with pairing.
\end{theorem}

It may seem surprising that the Generalized Riemann Hypothesis (GRH), one of the major unsolved problems in number theory, would play a role in a problem about critical groups of graphs.  The connection comes via the existence of small quadratic non-residues that satisfy additional properties.  In \cite{GJRWW}, the authors explain how a positive answer to the following conjecture would remove this dependence on GRH.
\begin{conjecture}
Let $p$ be a prime.  There exists a prime $q < 2\sqrt{p}$ with $q \equiv 3 \pmod{4}$ such that $q$ is a quadratic non-residue modulo $p$.
\end{conjecture}

Theorem \ref{thm_wood} gives the probability that the Sylow $p$-subgroup of the critical group of an Erd\H{o}s-R\'enyi random graph $G(n,q)$ is isomorphic to a particular finite abelian $p$-group.  Clancy, Leake and Payne give the analogous conjecture for a finite abelian $p$-group together with a perfect pairing \cite{CLP}.

\begin{conjecture}\label{conj_pairing}
Fix $0 < q < 1$.  Let $\Gamma$ be a finite abelian $p$-group and $\langle \cdot, \cdot \rangle$ be a perfect pairing on $\Gamma$.  Then, as $n$ goes to infinity, the probability that the Sylow $p$-subgroup of the critical group of the Erd\H{o}s-R\'enyi random graph $G(n,q)$ is isomorphic to $\Gamma$ with its associated monodromy pairing isomorphic to $\langle \cdot, \cdot \rangle$ is
\[
\frac{\prod_{i=1}^\infty (1-p^{1-2i})}{|\Gamma| \cdot |\Aut(\Gamma,\langle \cdot, \cdot \rangle)|},
\]
where $\Aut(\Gamma,\langle \cdot, \cdot \rangle)$ is the set of automorphisms of $\Gamma$ that preserve the pairing $\langle \cdot, \cdot \rangle$.
\end{conjecture}

We defined the critical group of a connected graph $G$ as the cokernel of its reduced Laplacian $L_0$, so we should also be able to understand the pairing on $K(G)$ in terms of this matrix.  In fact, this pairing is an instance of the pairing taking values in $\QQ/\ZZ$ defined on the cokernel of any nonsingular symmetric integer matrix $A$ induced by
\[
\langle x,y\rangle = y^T A^{-1} x.
\]
See \cite{CKLPW} for a discussion of the pairing on the cokernel of a symmetric matrix. In particular, Theorem 2 of \cite{CKLPW} shows that Conjecture \ref{conj_pairing} is consistent with Sylow $p$-subgroups of critical groups of random graphs being distributed like Sylow $p$-subgroups of cokernels of random elements of $\Sym_n(\ZZ_p)$ with their associated pairings.  Conjecture \ref{conj_pairing} is likely to be very difficult since it implies Theorem \ref{thm_wood}, the proof of which was a significant achievement that required the introduction of several new ideas into the study of critical groups.

\subsection{Ranks of Divisors and Gonality of Graphs}

We next introduce additional material about divisors on graphs that is motivated by connections to Brill--Noether theory, an important topic in algebraic geometry.  A divisor $\delta$ on $G$ is \emph{effective}\index{effective divisors} if $\delta(v) \ge 0$ for all $v$.  This property is not invariant under chip-firing. We have seen examples of divisors that are not effective but are chip-firing equivalent to divisors that are effective; for another example, see Figure~\ref{F:Effective}.

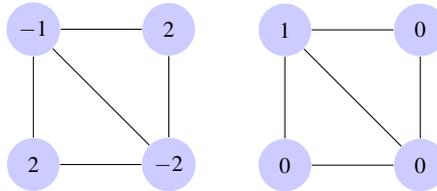
\begin{figure}
\begin{center}
\begin{tikzpicture}
  [scale=.3,auto=left,every node/.style={circle,fill=blue!20,minimum size=.7cm}]
  \node (l1) at (1,7) {$-1$};
  \node (l2) at (1,1) {$2$};
  \node (r1) at (7,7)  {$2$};
  \node (r2) at (7,1)  {$-2$};
\foreach \from/\to in {l1/r1,l2/r2,l1/l2,r1/r2,l1/r2}
    \draw (\from) -- (\to);
\end{tikzpicture}
\hskip .3in
\begin{tikzpicture}
  [scale=.3,auto=left,every node/.style={circle,fill=blue!20,minimum size=.7cm}]
  \node (l1) at (1,7) {$1$};
  \node (l2) at (1,1) {$0$};
  \node (r1) at (7,7)  {$0$};
  \node (r2) at (7,1)  {$0$};
\foreach \from/\to in {l1/r1,l2/r2,l1/l2,r1/r2,l1/r2}
    \draw (\from) -- (\to);
\end{tikzpicture}

\end{center}

\caption{Two divisors on the graph from Example \ref{E:run} that are chip-firing equivalent.  The first is not effective, but the second is.}
\label{F:Effective}
\end{figure}

A divisor $\delta$ has \emph{positive rank}\index{positive rank divisors} if for any $v \in V(G)$ the divisor $\delta'$ we get by setting $\delta'(v) = \delta(v) - 1$ and $\delta'(u) = \delta(u)$ for all other vertices $u$ is chip-firing equivalent to an effective divisor.  The \emph{gonality}\index{gonality} of $G$, denoted $\gon(G)$, is the smallest degree of an effective divisor with positive rank.

\begin{example}
Consider the following graph:
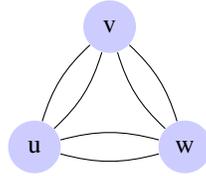
\begin{center}
\begin{tikzpicture}
  [scale=1,auto=left,every node/.style={circle,fill=blue!20,minimum size=.7cm}]
  \node (l) at (-1,0) {u};
  \node (u) at (0,1.6) {v};
  \node (r) at (1,0)  {w};
\foreach \from/\to in {l/u, u/l, l/r, r/l, u/r,r/u}
    \draw  (\from) to [bend left=15]  (\to);
\end{tikzpicture}
\end{center}

If $\delta$ is an effective divisor of degree one then we may assume without loss of generality that $\delta(u)=1$ and $\delta(v)=\delta(w)=0$.  One can show that the divisor $\delta'$ given by $\delta'(u)=1, \delta'(v)=-1, \delta'(w)=0$ is not equivalent to any effective divisor, which implies that $\delta$ does not have positive rank.  We will leave it as an exercise to show that no effective divisor of degree two has positive rank, either.  On the other hand, the divisor with $\delta(u)=\delta(v)=\delta(w)=1$ is a degree $3$ divisor of positive rank, showing that the gonality of this graph is $3$.
\end{example}

Several authors have studied ranks of divisors and the gonality of graphs.  For example, de Bruyn and Gijswijt connect the gonality of a graph to the notion of \emph{treewidth}, an important concept in graph theory \cite{deBruynGijswijt}.  The authors of \cite{DJKM} study the gonality of Erd\H{o}s-R\'enyi random graphs and prove the following theorem.

\begin{theorem}\cite[Theorem 1.1]{DJKM}\label{thm_expectedgon}
Let $p(n) = c(n)/n$, and suppose that $\log(n) \ll c(n) \ll n$.  Then the expected value of the gonality of an Erd\H{o}s-R\'enyi random graph $G(n,p(n))$ is asymptotic to $n$.
\end{theorem}
Related work of Amini and Kool in the setting of divisors of metric graphs leads to the similar results, but with bounds that are not as tight \cite{AminiKool}.

Theorem \ref{thm_expectedgon} gives the expected value of the gonality of one model of a random graph, but there are many other questions to consider.  Amini and Kool show in \cite{AminiKool} that random $d$-regular graphs on $n$ vertices have gonality bounded above and below by constant multiples of $n$.  Connections to tropical geometry led the authors of \cite{DJKM} to ask about the gonality of random $3$-regular graphs.  Dutta and Jensen prove a lower bound for the gonality of a regular graph $G$ in terms of the \emph{Cheeger constant} of $G$, one of the most studied measures of graph expansion  \cite{DuttaJensen}.  They also give a lower bound for gonality of a general graph $G$ in terms of its \emph{algebraic connectivity}\index{algebraic connectivity}, the second smallest eigenvalue of $L(G)$.  As a consequence they prove the following.
\begin{theorem}\cite[Theorem 1.3]{DuttaJensen}
Let $G$ be a random $3$-regular graph on $n$ vertices.  Then
\[
\gon(G) \ge 0.0072 n
\]
asymptotically almost surely.
\end{theorem}

\resproject{Can we improve the results about the expected gonality of a random $k$-regular graph?  What can we say about the expected gonality of other families of random graphs?}

There are several additional interesting directions in the Brill--Noether theory of graphs and metric graphs that have been the subject of successful research projects with undergraduate coauthors.  See for example, \cite{CDPR,KKW,LeakeRanganathan, LimPaynePotashnik}.

\subsection{Chip firing on Directed Graphs}\label{sec_directed}

Throughout this section, we have assumed that the graphs we consider are undirected. However, one can define a similar situation on directed graphs by considering the \emph{directed Laplacian matrix} $\hat{L}=D-A$, where $D$ is a diagonal matrix with $(i,i)$-entry equal to the outdegree of $v_i$, and the entries of the adjacency matrix $A$ correspond to the number of edges from $v_i$ to $v_j$.  The critical group of this directed graph is the torsion part of the cokernel of $\hat{L}$.

\begin{example}
Let us consider the following version of the graph from our running example where we consider some of the edges as being unidirectional:

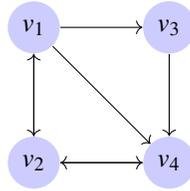
\begin{center}
\begin{tikzpicture}
  [scale=.3,auto=left,every node/.style={circle,fill=blue!20}]
  \node (l1) at (1,7) {$v_1$};
  \node (l2) at (1,1) {$v_2$};
  \node (r1) at (7,7)  {$v_3$};
  \node (r2) at (7,1)  {$v_4$};

\foreach \from/\to in {l1/r1,l2/r2,r2/l2,l2/l1,l1/l2,r1/r2,l1/r2}
    \path[->] (\from) edge (\to);
\end{tikzpicture}
\end{center}

The adjacency matrix, degree matrix, and directed Laplacian of this graph are given by:
\[
A = \left(
                                                                       \begin{array}{cccc}
                                                                         0 & 1 & 1 & 1 \\
                                                                         1 & 0 & 0 & 1 \\
                                                                         0 & 0 & 0& 1 \\
                                                                         0 & 1 & 0 & 0\\
                                                                       \end{array}
                                                                     \right),\ \ \ D = \left(
                                                                                                                       \begin{array}{cccc}
                                                                                                                         3 & 0 & 0 & 0 \\
                                                                                                                         0 & 2 & 0 & 0 \\
                                                                                                                         0 & 0 & 1 & 0 \\
                                                                                                                         0 & 0 & 0 & 1 \\
                                                                                                                       \end{array}
                                                                                                                     \right),\ \ \
                                                                     \hat{L} = \left(
                                                                                                                       \begin{array}{cccc}
                                                                                                                         3 & -1 & -1 & -1 \\
                                                                                                                         -1 & 2 & 0 & -1 \\
                                                                                                                         0 & 0 & 1 & -1 \\
                                                                                                                         0 & -1 & 0 & 1 \\
                                                                                                                       \end{array}
                                                                                                                     \right).\]
One can compute from the Smith Normal Form of $\hat{L}$ that $\cok(\hat{L}) \cong \ZZ$, so the associated critical group is trivial.
\end{example}

The notion of critical groups of directed graphs was first introduced in \cite{BL} and further developed in an unpublished note by Wagner \cite{Wag00}.  However, there are still many questions to be considered.

\resproject{Consider a finite connected undirected graph $G$.  For each edge of $G$ make a choice of how to orient it.  What can we say about the critical groups that occur as we vary over all possible choices?  For starters, consider the graph from the previous example.}

We can ask many of the questions considered in previous sections in this directed graph setting.  For example, for information on critical groups of Erd\H{o}s-R\'enyi random directed graphs see work of Koplewitz \cite{Koplewitz2} and Wood \cite{Wood2}.

\section{Arithmetical Structures}

In this section we consider a generalization of the Laplacian matrix and critical group of a graph that leads to interesting new enumerative problems.  The Laplacian of $G$ is defined by $L(G) = D-A$ where $A$ is the adjacency matrix of $G$ and $D$ is the diagonal matrix whose entries consist of the degrees of the vertices of the graph.  One generalization of this idea is to allow the entries on the diagonal of $D$ to be other positive integers. This leads to the notion of arithmetical structures, the topic of this section.

\subsection{Definitions and Examples}

Let $G$ be a finite connected graph with adjacency matrix $A$.  We define an \emph{arithmetical structure} \index{Arithmetical structures}on $G$ by a vector $\vd \in \ZZ_{\ge 0}^n$ so that there exists a vector $\vr \in \ZZ_{>0}^n$ with $(D-A)\vr =\vzero$, where $D$ is the diagonal matrix with the entries of $\vd$ along the diagonal.  We will sometimes write $D = \diag(\vd)$.

\begin{exercise}
In Section \ref{sec1}, we saw that for a connected graph $G$ with $|V(G)| = n$, the Laplacian matrix $L(G) = D-A$ has rank $n-1$.  Show that for any arithmetical structure on $G$, the matrix $\diag(\vd) - A$ has rank $n-1$.
\end{exercise}

This exercise shows that the null space of $\diag(\vd) - A$ is $1$-dimensional, so there is a unique vector in it up to scalar multiplication.  We typically choose $\vr$ to be the vector in $\Null(D-A)$ whose entries are all relatively prime positive integers.  This choice of $\vr$ uniquely specifies an arithmetical structure on $G$.  As such, we often refer to the pair $(\vr,\vd)$ as an arithmetical structure, even though each one is uniquely determined by the other.  We denote the matrix $\diag(\vd) - A$ by $L(G,\vr)$.  In Section \ref{sec1} we studied one arithmetical structure at length, $(\vone,\vd)$, where $\vd$ is the vector consisting of the degrees of the vertices of $G$.  This is the \emph{Laplacian arithmetical structure} on $G$.  In this case, $L(G,\vone) = L(G)$.

The $\vr$-vector of an arithmetical structure has another interpretation based on elementary number theory. In particular, one can think of an arithmetical structure as a labeling of the vertices of $G$ with relatively prime positive integers so that the label of any given vertex is a divisor of the (weighted, if necessary) sum of its neighbors.

\begin{example}\label{arithmeticalexample}
Consider again the situation from Example \ref{E:run}:
\[G=\begin{tikzpicture}
  [scale=.2,auto=left,every node/.style={circle,fill=blue!20},baseline={([yshift=-.8ex]current bounding box.center)}]
  \node (l1) at (1,7) {};
  \node (l2) at (1,1) {};
  \node (r1) at (7,7)  {};
  \node (r2) at (7,1)  {};
\foreach \from/\to in {l1/r1,l2/r2,l1/l2,r1/r2,l1/r2}
    \draw (\from) -- (\to);
\end{tikzpicture}, \hskip .5in A = \left(
                                                                       \begin{array}{cccc}
                                                                         0 & 1 & 1 & 1 \\
                                                                         1 & 0 & 0 & 1 \\
                                                                         1 & 0 & 0& 1 \\
                                                                         1 & 1 & 1 & 0\\
                                                                       \end{array}
                                                                     \right).
                                                                     \]
\noindent Let $\vd = \left(
                                                                       \begin{array}{cccc}
                                                                         5&6&3&1 \\
                                                                       \end{array}
                                                                     \right)^T$. The null space of the matrix
\[L(G,\vr) = D-A =  \left(
                                                                       \begin{array}{cccc}
                                                                         5 & -1 & -1 & -1 \\
                                                                         -1 & 6 & 0 & -1 \\
                                                                         -1 & 0 & 3& -1 \\
                                                                         -1 & -1 & -1 & 1\\
                                                                       \end{array}
                                                                     \right)\]
is spanned by the vector $ \vr = \left(\begin{array}{cccc}
                                                                         3&2&4&9 \\
                                                                       \end{array}\right)^T$, so $(\vr,\vd)$ is an arithmetical structure on $G$.  If we label the graph  as below then the label of each vertex is a divisor of the sum of the labels of its neighbors.

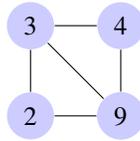
\begin{center}
\begin{tikzpicture}
  [scale=.2,auto=left,every node/.style={circle,fill=blue!20}]
  \node (l1) at (1,7) {3};
  \node (l2) at (1,1) {2};
  \node (r1) at (7,7)  {4};
  \node (r2) at (7,1)  {9};
\foreach \from/\to in {l1/r1,l2/r2,l1/l2,r1/r2,l1/r2}
    \draw (\from) -- (\to);
\end{tikzpicture}
\end{center}

\end{example}

\begin{exercise}\label{Ex:arith1}
Find more arithmetical structures on the graph from this example.  As a hint, there are a total of $63$ structures, and the largest entry of any $\vr$ that occurs is $18$.
\end{exercise}

Just as we defined the critical group of a graph $G$ to be the torsion part of the cokernel of $L(G)$, we can define the critical group associated to any arithmetical structure $(\vr,\vd)$ to be the torsion part of the cokernel of $L(G,\vr)$.  We denote this critical group by $\mathcal{K}(G;\vr)$.  We described how to compute $\cok(L(G))$ by finding its Smith Normal Form and can proceed similarly in the more general setting with the matrix $L(G,\vr)$.  If we do this for the matrix from Example \ref{arithmeticalexample}, we see that the associated critical group is trivial.  In Section \ref{SS:ArithCrit} we will analyze the structure of this group in more depth.

The concept of arithmetical structures on graphs was originally developed by Lorenzini in \cite{Lor89} as a way of trying to understand the N\'eron models of certain algebraic curves where components might appear with multiplicity greater than one.  Explaining these applications is beyond the scope of this note, but we refer the interested reader to \cite{Lor90}.  We also refer the reader to \cite[Section 4]{AsadiBackman} where Asadi and Backman show that chip-firing on arithmetical graphs can be interpreted as a special case of the chip-firing for directed multigraphs that we introduced in Section \ref{sec_directed}, but do not pursue this perspective further here.

\subsection{Counting Arithmetical Structures}

In \cite{Lor89} Lorenzini proves that any finite connected graph has a finite number of arithmetical structures.  However, the proof is nonconstructive and in general does not give an upper bound for the number of these arithmetical structures.  In recent years, several authors have become interested in trying to count the number of arithmetical structures on certain types of graphs.

One general approach to counting arithmetical structures comes from the following observation.  We first introduce some notation.  Let $G$ be a graph and $(\vr,\vd)$ be an arithmetical structure on $G$.  For $v \in V(G)$ we write $\vr_v$ for the value of $\vr$ corresponding to $v$ and $\vd_v$ for the value of $\vd$ corresponding to $v$.
\begin{theorem}
Let $G$ be a graph and let $(\vr,\vd)$ be an arithmetical structure on $G$.  Assume that $v$ is a vertex of degree $2$ with neighbors $u$ and $w$ so that $\vr_v > \vr_u$ and $\vr_v > \vr_w$.  Then $\vr_v=\vr_u+\vr_w$.

Moreover, if one defines the graph $G'$ to be the graph whose vertex set is $V(G')=V(G) \setminus \{v\}$, and whose edge set is $E(G') = E(G) \cup \{\overline{uw}\} \setminus \{\overline{uv}, \overline{vw}\}$ then one gets a new arithmetical structure on $G'$ by defining $\vr'$ to have the same values as $\vr$ on all remaining vertices.
\end{theorem}

\begin{exercise}
Verify that this theorem holds for the structures that you found in Exercise \ref{Ex:arith1}.
\end{exercise}

\begin{proof}
The proof of the first claim follows from the fact that if we have such an arithmetical structure we know that $\vr_v \mid (\vr_u+\vr_w)$.  If we know that $\vr_v > \vr_u$ and $\vr_v > \vr_w$, then $\vr_u + \vr_w<2\vr_v$, which implies that $\vr_u+\vr_w=\vr_v$.

The proof of the second claim is straightforward and can be best understood by considering a picture such as the one in Figure \ref{F:smooth2}, and making the observation that if $\vr_u \mid \left((\vr_u+\vr_w)+\sum \vr_i\right)$ then $\vr_u \mid (\vr_w + \sum \vr_i)$.
\end{proof}

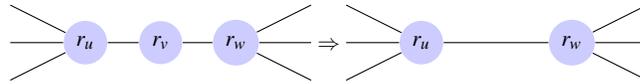
\begin{figure}
\begin{center}
\raisebox{-.4\height}{\begin{tikzpicture}
  [scale=1,auto=left,every node/.style={circle,fill=blue!20}]
  \node (u) at (1,0) {$r_u$};
  \node (v) at (2,0) {$r_v$};
  \node (w) at (3,0)  {$r_w$};
  \coordinate (u1) at (0,.5);
  \coordinate (u2) at (0,0);
  \coordinate (u3) at (0,-.5);
  \coordinate (v1) at (4,.5);
  \coordinate (v2) at (4,0);
  \coordinate (v3) at (4,-.5);
\foreach \from/\to in {u/v,v/w,u/u1,u/u2,u/u3,w/v1,w/v2,w/v3}
    \draw (\from) -- (\to);
\end{tikzpicture}}
$\Rightarrow$
\raisebox{-.4\height}{\begin{tikzpicture}
  [scale=1,auto=left,every node/.style={circle,fill=blue!20}]
  \node (u) at (1,0) {$r_u$};
  \node (w) at (3,0)  {$r_w$};
  \coordinate (u1) at (0,.5);
  \coordinate (u2) at (0,0);
  \coordinate (u3) at (0,-.5);
  \coordinate (v1) at (4,.5);
  \coordinate (v2) at (4,0);
  \coordinate (v3) at (4,-.5);
\foreach \from/\to in {u/w,u/u1,u/u2,u/u3,w/v1,w/v2,w/v3}
    \draw (\from) -- (\to);
\end{tikzpicture}}
\end{center}
\caption{Pictures showing the `smoothing' operation at a vertex of degree two}
\label{F:smooth2}
\end{figure}

We refer to the operation of removing a vertex of degree $2$ corresponding to a local maximum of $\vr$, such as the one described in the previous theorem, as \emph{smoothing at vertex $v$}\index{smoothing}.  One can also define a smoothing operation at a vertex of degree $1$; in particular, if $v$ is a vertex of degree $1$ that is adjacent to the vertex $u$ and if $\vr_v = \vr_u$ then one gets a new arithmetical structure on a smaller graph by removing the vertex $v$, as illustrated in Figure \ref{F:smooth1}.  An arithmetical structure $(\vr,\vd)$ on $G$ is \emph{smooth}\index{smooth arithmetical structures} if there are no vertices of $G$ at which we can apply a smoothing operation.

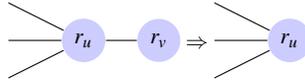
\begin{figure}
\begin{center}
\raisebox{-.4\height}{\begin{tikzpicture}
  [scale=1,auto=left,every node/.style={circle,fill=blue!20}]
  \node (u) at (1,0) {$r_u$};
  \node (v) at (2,0) {$r_v$};

  \coordinate (u1) at (0,.5);
  \coordinate (u2) at (0,0);
  \coordinate (u3) at (0,-.5);

\foreach \from/\to in {u/v,u/u1,u/u2,u/u3}
    \draw (\from) -- (\to);
\end{tikzpicture}}
$\Rightarrow$
\raisebox{-.4\height}{\begin{tikzpicture}
  [scale=1,auto=left,every node/.style={circle,fill=blue!20}]
   \node (u) at (1,0) {$r_u$};
   \coordinate (u1) at (0,.5);
  \coordinate (u2) at (0,0);
  \coordinate (u3) at (0,-.5);
\foreach \from/\to in {u/u1,u/u2,u/u3}
    \draw (\from) -- (\to);
\end{tikzpicture}}
\end{center}
\caption{Pictures showing the `smoothing' operation at a vertex of degree $1$}
\label{F:smooth1}
\end{figure}

These smoothing operations are reversible, and in particular the number of ways that one can take an arithmetical structure and \emph{subdivide} it can be described in terms of certain ballot numbers. (For details, see \cite{Oax} and \cite{ICERM}).  The approach taken in those papers is to count the number of smooth structures on smaller graphs and then count the number of ways they can be subdivided into general arithmetical structures on $G$.  In particular, one can show theorems of the following type:

\begin{theorem}
We can count the number of smooth structures on certain graphs in the following way:
\begin{enumerate}
\item The only smooth structure on a path is given by the Laplacian arithmetical structure on a single vertex.  The total number of structures on a path of length $n$ is given by the $(n-1)$\textsuperscript{st} Catalan number\index{Catalan number}, $C_{n-1} = \frac{1}{n} \binom{2(n-1)}{n-1}$.
\item The only smooth structure on a cycle of length $n$ is given by the Laplacian arithmetical structure. The total number of structures on a cycle on $n$ vertices is given by the binomial coefficient $\binom{2n-1}{n-1}$.
\item Let $n \ge 4$ and $P_n'$ be the path graph on $n$ vertices where the first edge is doubled. The number of smooth structures on $P_n'$ is $4$, and the total number of structures is $4C_{n-1}-2C_{n-2}$.
\end{enumerate}
\end{theorem}

In general it appears to be quite difficult to count precisely the number of smooth arithmetical structures on a graph.  For example, even for a \emph{bident} graph, a path plus one additional vertex connected only to the second vertex on the path, it is only known that the number of smooth arithmetical structures is bounded between two cubic polynomials in the number of vertices \cite{ICERM}.

\resproject{
Consider the graph $\widetilde{C_4}$ obtained by taking the cycle $C_4$ and adding a second edge between two consecutive vertices.
\begin{enumerate}
\item How many smooth arithmetical structures are there on $\widetilde{C_4}$?
\item How many total arithmetical structures are there on $\widetilde{C_4}$?
\item What if we instead consider bigger cycles or add more edges?
\end{enumerate}
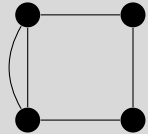
\begin{center}
\begin{tikzpicture}
[scale=.7,auto=left,every node/.style={circle,fill=black}]
\node (v1) at (0,0) {};
\node (v2) at (2,0) {};
\node (v3) at (2,2) {};
\node (v4) at (0,2) {};
\draw (v1) -- (v2);
\draw (v3) -- (v2);
\draw (v3) -- (v4);
\draw (v1) to [bend left] (v4);
\draw (v4) -- (v1);
\end{tikzpicture}
\end{center}
}

In the definition of smoothing at a vertex $v$ of degree $2$ or $1$, we have $\vd_v = 1$.  One might wonder whether this idea could be generalized to vertices $v$ of larger degree at which $\vd_v = 1$.  These smoothing operations are special cases of the \emph{clique-star transform}\index{clique-star transform} defined in \cite{CV2}. This operation replaces a subgraph that is isomorphic to a star graph, the complete bipartite graph $K_{1,n}$, by the complete graph on $n$ vertices.

As an example, let us consider arithmetical structures on the complete graph $K_n$.  Every such arithmetical structure is uniquely determined by a vector of relatively prime positive integers $\vr = (r_1,\ldots,r_n)$ where each $r_i$ divides the sum  $\sum_{i=1}^n r_i$.  The star graph\index{star graph} $K_{1,n}$ consists of a vertex $v_0$ connected to $n$ other vertices, each of which has degree $1$.  If an arithmetical structure on this graph has $\vd_{v_0} = 1$, then $r_0 = \sum_{i=1}^n r_i$.  Therefore, such arithmetical structures on $K_{1,n}$ are in bijection with the set of all arithmetical structures on $K_n$.  It is interesting to further consider the remaining structures on $K_{1,n}$ that have $\vd_{v_0}>1$.

\begin{figure}
\begin{center}

\begin{tikzpicture} [scale=1,auto=left,every node/.style={circle,fill=blue!20}]
\node (lu) at (-1,.75) {$v_2$};
\node (ld) at (-1,-.75) {$v_1$};
\node (u) at (0,1.3) {$v_3$};
\node (ru) at (1,.75) {$v_4$};
\node (rd) at (1,-.75) {$v_5$};
\node (d) at (0,-1.3) {$v_6$};
\foreach \from/\to in {lu/ld,lu/u,lu/ru,lu/rd,lu/d,ld/u,ld/ru,ld/rd,ld/d,u/ru,u/rd,u/d,ru/rd,ru/d,rd/d}
    \draw (\from) -- (\to);
\end{tikzpicture}
\hskip .4in
\begin{tikzpicture} [scale=1,auto=left,every node/.style={circle,fill=blue!20}]
\node (lu) at (-1,.75) {$v_2$};
\node (ld) at (-1,-.75) {$v_1$};
\node (u) at (0,1.3) {$v_3$};
\node (ru) at (1,.75) {$v_4$};
\node (rd) at (1,-.75) {$v_5$};
\node (d) at (0,-1.3) {$v_6$};
\node (c) at (0,0) {$v_0$};
\foreach \from/\to in {lu/c,ld/c,u/c,ru/c,rd/c,d/c}
    \draw (\from) -- (\to);
\end{tikzpicture}

\caption{The complete graph $K_6$ and the star graph $K_{1,6}$}
\label{F:cliquestar}
\end{center}
\end{figure}
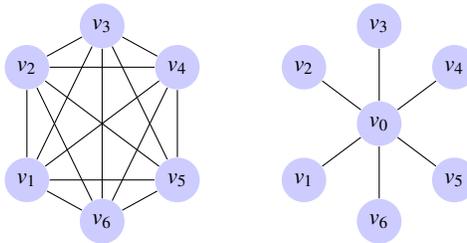

To further consider the set of arithmetical structures on $K_n$, we note that the definition of an arithmetical structure implies that for each~$i$:
\begin{eqnarray*}
d_ir_i &=& \sum_{j \ne i} r_j \\
(d_i+1)r_i &=& \sum_j r_j \\
\frac{1}{d_i+1}&=& \frac{r_i}{\sum_j r_j}.
\end{eqnarray*}
In particular, if we sum over all $i$ we get that
\[
\displaystyle \sum_{i=1}^n \frac{1}{d_i+1}=1.
\]
The arithmetical structures of $K_n$ are therefore in bijection with ways of writing $1$ as a sum of reciprocals of $n$ positive integers.  Finding the number of ways of doing this is a difficult problem in additive number theory.

\begin{exercise}
Classify all sets of positive integers $\{a_1,a_2,a_3,a_4\}$ so that $\sum \frac{1}{a_i} = 1$.  For each one find the corresponding arithmetical structure on $K_4$.
\end{exercise}

In general, there is no known formula for this number, but we do have a lower bound that is doubly exponential in $n$ \cite{Egyptian}.  Corrales and Valencia get similar results for all structures on star graphs \cite{CV}.  We close this section with a conjecture from \cite{CV2}, which is based on the observation that vertices with higher degree seem to lead to more arithmetical structures.

\resproject{Show that for any simple connected graph $G$ with $n$ vertices the number of arithmetical structures on $G$ is at least the number on the path $P_n$ and at most the number on the complete graph $K_n$}

\subsection{Critical Groups of Arithmetical Structures}\label{SS:ArithCrit}

We have already seen that is difficult to enumerate all arithmetical structures on a given graph.  However, it might be easier to say something about the critical groups that occur associated to this set of arithmetical structures.  For example, it is shown in \cite{Oax} that every arithmetical structure on a path leads to a trivial critical group; we will give an alternative proof of this fact below.  Recall that we define the critical group $\mathcal{K}(G;\vr)$ of an arithmetical structure $(\vd,\vr)$ to be the torsion part of the cokernel of $L(G,\vr)=\diag(\vd)-A$.

To set our notation, let $G$ be a finite multigraph with $V(G) = \{v_1,\ldots, v_n\}$.  Let $x_{i,j}$ be the number of edges between $v_i$ and $v_j$.  Since $G$ is a multigraph we note that $x_{i,j}$ may be larger than $1$.  Let $(\vr,\vd)$ be an arithmetical structure on $G$.  We define $G_\vr$ to be the graph with the same vertex set, $V(G)$, and with $x_{i,j}r_i r_j$ edges between any two vertices $v_i$ and  $v_j$.  We leave the proof of the following lemma as an exercise in linear algebra:

\begin{lemma}
We have $L(G_\vr,\vone) = R L(G,\vr)R$, where $R = \diag(\vr)$.
\end{lemma}

Let $L(G,\vr)^{i,j}$ be the matrix we get from $L(G,\vr)$ by deleting its $i$\textsuperscript{th} row and $j$\textsuperscript{th} column.  Similar to the situation in Corollary \ref{C:order}, the determinant of $L(G,\vr)^{i,j}$ is given by $r_ir_j|\mathcal{K}(G;\vr)|$.  From this, one can compute:
\begin{eqnarray*}
|\mathcal{K}(G_\vr;\vone)| &=& \det(L(G_\vr,\vone)^{1,1}) \\
&=&\det((R L(G,\vr)R)^{1,1}) \\
&=&\det(R^{1,1})\det(L(G,\vr)^{1,1})\det(R^{1,1}) \\
&=&(r_2\ldots r_n)^2 r_1^2 |\mathcal{K}(G;\vr)|.
\end{eqnarray*}

On the other hand, we know from Corollary \ref{C:Kirchhoff} that $|\mathcal{K}(G_\vr;\vone)|$ is the number of spanning trees of $G_\vr$.  So, we can determine the order of $\mathcal{K}(G;\vr)$ by counting spanning trees of the graph $G_\vr$.

Let us first consider the special case where the skeleton of $G$ is a tree.  Let $V(G) = \{v_1,\ldots, v_n\}$.  The \emph{skeleton}\index{skeleton} of a multigraph $G$ is the graph $\overline{G}$ that has the same vertex set as $G$ and has $\min(1,x_{i,j})$ edges between any pair of vertices $v_i,v_j$.  Intuitively, this is what happens when you remove all `repeated' edges.  If $\overline{G}$ is a tree then it is easy to see that the number of spanning trees of $G$ is $\displaystyle \prod_{x_{i,j} \ne 0} x_{i,j}$.  Moreover, it is clear that $\overline{G_\vr}$ is also a tree and therefore that the number of spanning trees of $G_\vr$ is given by
\[
|\mathcal{K}(G_\vr;\vone)| = \prod_{x_{i,j} \ne 0} x_{i,j}r_ir_j=\prod_{x_{i,j} \ne 0} x_{i,j} \prod_{i=1}^n r_i^{\deg(v_i)}.
\]
  In particular, this proves the following result of Lorenzini \cite[Corollary 2.3]{Lor89}.

\begin{corollary}\label{C:orderarith}
Let $G$ be a graph with $V(G) = \{v_1,\ldots, v_n\}$ so that $\overline{G}$ is a tree and let $(\vd,\vr)$ be an arithmetical structure on $G$.  Then
\[
|\mathcal{K}(G;\vr)| = \prod_{x_{i,j} \ne 0} x_{i,j} \prod_{i=1}^n r_i^{\deg(v_i)-2}.
\]
\end{corollary}

More generally, one can count spanning trees of $G_\vr$ by noting that a spanning tree of $G$ that includes an edge $\overline{v_i v_j}$ will lead to $r_i r_j$ spanning trees in $G_{\vr}$, as we can choose any of the related edges.  In particular, a spanning tree $\CT$ of $G$ leads to $\prod_i r_i^{\deg_\CT(v_i)}$ spanning trees of $G_\vr$, where $\deg_\CT(v_i)$ denotes the degree of the vertex $v_i$ in the tree $\CT$.  This discussion proves the following theorem:

\begin{theorem}\label{T:count}
Let $G$ be a graph with $V(G) = \{v_1,\ldots, v_n\}$ and let $\vr$ give an arithmetical structure on $G$.  Then we have
\[
|\mathcal{K}(G;\vr)| = \sum_{\CT \subseteq G} \left(\prod_{i=1}^n r_i^{\deg_\CT(v_i)-2}\right).
\]
where the sum ranges over all spanning trees of the graph $G$.
\end{theorem}

\begin{example}
If $G$ is a path on $n$ vertices then there is a single spanning tree given by $G$ itself.  It follows from \cite[Lemma 1]{Oax} that any arithmetical structure on a path has $r_1=r_n=1$. Therefore one computes that
\begin{eqnarray*}
|\mathcal{K}(G;\vr)| &=& \sum_{\CT \subseteq G} \left(\prod_{i=1}^n r_i^{\deg_\CT(v_i)-2}\right)\\
&=& \prod_{i=1}^n r_i^{\deg_G(v_i)-2} =\frac{1}{r_1r_n} = 1.
\end{eqnarray*}
This gives an alternative proof to the first claim in \cite[Theorem 7]{Oax}.
\end{example}

\begin{example}
Let $G$ be a cycle on $n$ vertices. A spanning tree of $G$ corresponds to removing a single edge.  In particular, Theorem \ref{T:count} implies that
\[
|\mathcal{K}(G;\vr)| = \sum_{i=1}^n \frac{1}{r_i r_{i+1}}.
\]
If $\vr \ne \vone$ then the arithmetical structure has some vertex $v_i$ with $r_i=r_{i-1}+r_{i+1}$, so we can smooth the structure at this vertex.  In particular, we note that
\[
\frac{1}{r_{i-1}r_i}+\frac{1}{r_ir_{i+1}} = \frac{1}{r_{i-1}(r_{i-1}+r_{i+1})} + \frac{1}{r_{i+1}(r_{i-1}+r_{i+1})} = \frac{1}{r_{i-1}r_{i+1}}.
\]
This shows us that smoothing the structure at this vertex will not change the order of the critical group.  Any arithmetical structure $(\vr,\vd)$ on $C_n$ can be smoothed to the Laplacian arithmetical structure on some $C_k$ with $k \le n$.  For this value of $k$ we see that
\[
|\mathcal{K}(C_n;\vr)| = |\mathcal{K}(C_k;\vone)| = k.
\]
\end{example}

Understanding the structure of the group $\mathcal{K}(G;\vr)$ rather than just its order requires a more careful analysis. The following theorem is a restatement of Proposition 1.12 in \cite{Lor89}.

\begin{theorem}
We have the following two short exact sequences:
\[
1 \rightarrow \bigoplus \ZZ/r_i\ZZ \rightarrow E \rightarrow \mathcal{K}(G;\vr) \rightarrow 1
\]
\[
1 \rightarrow E \rightarrow \mathcal{K}(G_\vr;\vone) \rightarrow \bigoplus \ZZ/r_i\ZZ \rightarrow 1
\]
where $E$ is a specific quotient group.
\end{theorem}

In general, these short exact sequences do not split but they do give us insight about the structure of $\mathcal{K}(G;\vr)$ if we know the structure of $\mathcal{K}(G_\vr;\vone)$.

\resproject{What are the possible critical groups associated to a given graph $G$ as we vary the arithmetical structures $(\vr,\vd)$?}

Answers to this question are known only in a few cases.  We have already seen what happens with paths and cycles.  Critical groups associated to arithmetical structures on bident graphs $D_n$ are analyzed in \cite[Section 5]{ICERM}.  In particular, the authors show that for any $\vr$, the matrix $L(G,\vr)$ has an $(n-2) \times (n-2)$ minor equal to $1$ and use Corollary \ref{C:cyclic} to show that $\mathcal{K}(G;\vr)$ is cyclic.  An analysis similar to the one leading to Corollary \ref{C:orderarith} shows that the biggest possible order will be $2n-5$ and completely characterizes the smaller critical group orders that occur.

There are natural generalizations of many of the problems from Section \ref{sec1} to arithmetical graphs.  For example, see \cite[Section 5]{BoschLorenzini} for results on a realization problem for arithmetical graphs.

\begin{acknowledgement}
We would like to thank Luis David Garcia-Puente for initiating this project.  We would further like to thank David Jensen, Pranav Kayastha, Sam Payne, Farbod Shokrieh, and the editors and referees for their helpful comments.

The second author is supported by NSF Grant DMS 1802281.
\end{acknowledgement}

\bibliographystyle{amsplain}

\providecommand{\bysame}{\leavevmode\hbox to3em{\hrulefill}\thinspace}
\providecommand{\MR}{\relax\ifhmode\unskip\space\fi MR }
\providecommand{\MRhref}[2]{%
  \href{http://www.ams.org/mathscinet-getitem?mr=#1}{#2}
}
\providecommand{\href}[2]{#2}

\noindent {{\bf Note:} We have marked papers that have at least one undergraduate coauthor \ugrad{in red}.
\end{document}